\documentclass[11pt,a4paper]{scrartcl}

\makeatletter
\DeclareOldFontCommand{\rm}{\normalfont\rmfamily}{\mathrm}
\DeclareOldFontCommand{\sf}{\normalfont\sffamily}{\mathsf}
\DeclareOldFontCommand{\tt}{\normalfont\ttfamily}{\mathtt}
\DeclareOldFontCommand{\bf}{\normalfont\bfseries}{\mathbf}
\DeclareOldFontCommand{\it}{\normalfont\itshape}{\mathit}
\DeclareOldFontCommand{\sl}{\normalfont\slshape}{\@nomath\sl}
\DeclareOldFontCommand{\sc}{\normalfont\scshape}{\@nomath\sc}
\makeatother

\usepackage{placeins}
\usepackage{float}
\usepackage{a4wide}
\usepackage{latexsym}
\usepackage[USenglish]{babel}
\usepackage[utf8]{inputenc}
\usepackage[T1]{fontenc}
\usepackage[breaklinks=true]{hyperref}
\usepackage{amsmath,amssymb}
\usepackage[amsthm,thmmarks]{ntheorem}
\usepackage{hhline}
\usepackage{float}
\usepackage{mathrsfs}
\usepackage{graphicx}
\usepackage{amsmath}
\usepackage{mathdots}
\usepackage{amsfonts}
\usepackage{amsxtra}
\usepackage{url}
\usepackage{mathptmx}
\usepackage{helvet}
\usepackage{stmaryrd}
\usepackage{dsfont}
\usepackage{enumerate}
\usepackage{MathDef}
\usepackage{verbatim}
\usepackage{color}
\usepackage{cleveref}
\usepackage{cancel}
\usepackage[authoryear,round]{natbib}

\def\C{\mathbb{C}}
\def\E{\mathbb{E}}
\def\N{\mathbb{N}}
\def\P{\mathbb{P}}

\def\R{\mathbb{R}}
\def\Z{\mathbb{Z}}

\def\1{\mathds{1}}

\def\f{f_Y^{(\Delta)}}

\def\pa2{\frac{\partial ^2}{\partial \vartheta ^2}}

\def\sc{\textsc}
\def\bf{\textbf}

\theoremstyle{plain}
\newtheorem{theorem}{Theorem}[section]

\newtheorem{proposition}[theorem]{Proposition}
\newtheorem{definition}[theorem]{Definition}

\newtheorem{corollary}[theorem]{Corollary}
\newtheorem{assumptionletter}{{\textbf{Assumption}}}

\theoremstyle{definition}
\newtheorem{example}[theorem]{Example}
\newtheorem{remark}[theorem]{Remark}

\numberwithin{equation}{section}
\allowdisplaybreaks

\renewcommand{\labelenumi}{(\roman{enumi})}
\title{A note on estimation  of $\alpha$-stable\\[2mm] CARMA processes sampled\\[2mm] at  low frequencies}

\author{Vicky Fasen-Hartmann \setcounter{footnote}{1}\thanks{Institute of Stochastics, Englerstra{\ss}e 2,
D-76131 Karlsruhe, Germany. \emph{Email:}
\href{mailto:vicky.fasen@kit.edu}{vicky.fasen@kit.edu}}
\and Celeste Mayer \setcounter{footnote}{0}\thanks{Institute of Stochastics, Englerstra{\ss}e 2,
D-76131 Karlsruhe, Germany. \emph{Email:}
\href{mailto:celeste.mayer@kit.edu}{celeste.mayer@kit.edu}}}

\date{}

\begin{document}
%
%%%%%%%%%%%%%%%%%  Titel und Autoreninformationen  %%%%%%%%%%%%%%%%%%%%%%%%%%%%%%%%
\maketitle
%
%%%%%%%%%%%%%%%%%  Abstract und AMS Classification  %%%%%%%%%%%%%%%%%%%%%%%%%%%%%%%
\begin{abstract}
In this paper, we investigate estimators for symmetric $\alpha$-stable CARMA processes sampled equidistantly. Simulation studies suggest
that the Whittle estimator and the estimator presented in \cite{Garcia:Klu:Mueller:2011} are consistent estimators for the
parameters  of stable CARMA processes.
For CARMA processes with finite second moments it is well-known that the Whittle estimator is consistent and asymptotically normally distributed. Therefore, in the light-tailed setting the properties of the Whittle estimator for CARMA processes are similar to those of the Whittle estimator for ARMA processes. However, in the present paper
we prove that, in
general, the Whittle estimator for symmetric $\alpha$-stable CARMA processes sampled at low frequencies is  not consistent and highlight why simulation studies suggest
something else. Thus, in contrast to the light-tailed setting the properties of the Whittle estimator for heavy-tailed ARMA processes can not be transferred to heavy-tailed CARMA processes.
We elaborate as well that
the estimator presented in \cite{Garcia:Klu:Mueller:2011} faces the same problems.
However, the Whittle estimator for stable CAR(1) processes is consistent.

\end{abstract}

%we consider an adjusted Whittle estimator for the parameters of an equidistantly sampled CARMA process with a symmetric $\alpha$-stable L\'evy process as driving process. For light-tailed CARMA processes and heavy-tailed ARMA processes, it is already known that the Whittle function converges to a deterministic function which has a global minimum in the true parameter. In contrast, we show that in the $\alpha$-stable CARMA setting the adjusted Whittle function converges to a stochastic process which does not necessarily has a minimum in the true parameter. Essential that the Whittle estimator might fail for $\alpha$-stable CARMA processes is that the white noise of the sampled process is not independent and identically distributed in general. As an exception, the white noise of a sampled Ornstein-Uhlenbeck process is independent and identically distributed. For these processes we show that the adjusted Whittle estimator is consistent. Our theoretical results are demonstrated through a simulation study.

\noindent
\begin{tabbing}
\emph{AMS Subject Classification 2010: }\=Primary:  62M86,  62F12, 62M10
\\ \> Secondary:  60G10, 62F10
\end{tabbing}

\vspace{0.2cm}\noindent\emph{Keywords:} autocovariance function, CARMA process, consistency,   periodogram, Ornstein-Uhlenbeck process,
stable Lévy process, state space model, Whittle estimator

	\section{Introduction}

Continuous-time ARMA (CARMA) processes are the continuous-time versions of the well-known ARMA processes
in discrete time. Since CARMA processes with finite second moments sampled equidistantly are weak ARMA processes,
several methods to estimate the parameters of ARMA processes as, e.g., quasi-maximum likelihood estimators or Whittle
estimators,
can be utilized to estimate the parameters of CARMA processes. A challenge is that the distribution of the  white noise
of the weak ARMA representation
is highly dependent on the model parameters of the CARMA process and only a weak white noise instead of a strong white noise. Moreover, some attention has to be paid
to guarantee the model identifiability.  These generalisations are technical and take some effort.
 \cite{QMLE} derived a rigorous theory
for quasi maximum-likelihood estimation for equidistantly sampled multivariate CARMA processes with finite second moments; see as well
  \cite{Brockwell:Lindner:2019} and \cite{brockwelldavisyang} for CARMA processes. The
method itself is much older and applied in several disciplines, in particular, in economics, even though the statistical inference
of that estimator was as far as we know not investigated earlier. Since this procedure essentially depends on the variance
of the sampled white noise, it is unsuitable for estimating the parameters of heavy-tailed CARMA processes. The consistency and the asymptotic normality of the Whittle
estimator for multivariate CARMA processes with finite second moments were recently derived  in \cite{Fasen:Mayer:2020}.
The limit covariance matrices of these estimators are slightly different to those
of the corresponding  ARMA estimators. Apart from that, the structure of the estimator is the same taking
the identifiability assumptions into account.

The topic of this paper is the estimation of symmetric $\alpha$-stable CARMA processes with infinite second moments sampled
at low frequencies. To the best of our knowledge there exist no estimators for heavy-tailed CARMA processes in the literature yet. Only
stable Ornstein-Uhlenbeck processes, which correspond to the class of CARMA(1,0) processes, are investigated, e.g., in \cite{Hu:Long:2007,Hu:Long:2009},
\cite{Fasen:2013} and \cite{Ljungdahl:Podolskij}.
 \cite{Garcia:Klu:Mueller:2011} proposed an indirect quasi-maximum likelihood method for stable CARMA processes. Their simulation study suggests
 that the estimator is converging  when the number of observations tends to infinity, however they do not present a mathematical analysis
 of their estimator. In this paper
 we mainly investigate the Whittle estimator in more detail.
 The Whittle estimator was first introduced by \cite{whittle1953estimation} to estimate
 the parameters of Gaussian ARMA processes and
 further explored in \cite{Hannan73},
 see the monograph of \cite{Brockwell:Davis:1991}. The more general case of Whittle estimation of multivariate ARMA processes
 with finite second moments was topic of \cite{DunsmuirHannan76}. \cite{Mikosch:Gardrich:Klueppelberg:1995} observed that
 the same method also works for symmetric $\alpha$-stable ARMA processes with infinite second moments.
 This is not apparent since the spectral density is not defined for processes with
 infinite second moments and the Whittle estimator
 is based on its empirical version, the periodogram.
 \cite{Mikosch:Gardrich:Klueppelberg:1995} derived the consistency and the convergence of the properly
 normalized and standardized Whittle estimator
 to a functional of stable random variables. Therefore, Whittle estimation is an estimation method for both heavy-tailed and light-tailed ARMA processes.
  %so that we can use the estimation packages for light-tailed ARMA processes for heavy-tailed
 %ARMA processes as well.
 The estimator is consistent in both settings which is a very important
 property. Indeed, the convergence rate of the Whittle estimator
 is $n^{1/\alpha}$  in the stable case which is even faster than in the case with finite second moments,
 where it is $n^{1/2}$. It would thus be desirable to have the same for heavy and light-tailed
 CARMA processes.

Similarly, for heavy-tailed fractional ARIMA processes, which exhibit
 long range dependence, the Whittle estimator is consistent and the asymptotic behaviour is known (see  \cite{Kokoszka:Taqqu:1996}).
The Whittle estimator was also used for parameter estimation of GARCH processes in \cite{Giraitis:Robinson:2001}
and \cite{Mikosch:Straumann:2002}).
For GARCH(1,1) processes \cite{Mikosch:Straumann:2002} pointed out that the Whittle estimator is
consistent as long as the 4th moment is finite and inconsistent when the 4th moment is infinite.
This is a very interesting statement to which we come back later.

Therefore, the behaviour of the Whittle estimator for  heavy-tailed ARMA
 models and the fact that the Whittle estimator also works for light-tailed
CARMA processes (see \cite{Fasen:Mayer:2020})
 might suggest that the Whittle estimator is suitable for parameter estimation of  heavy-tailed CARMA processes.
 In particular, there exist simulation studies which confirm this idea.
However, the main statement of the present paper is that the Whittle estimator for symmetric $\alpha$-stable CARMA processes
is in general not a consistent estimator, even though there are simulation experiments which suggest something else.
There is also evidence that the estimator of \cite{Garcia:Klu:Mueller:2011} is not consistent as well. We elaborate that in more detail and present several arguments
for our conjecture.

The paper is structured in the following way. We start with an introduction on symmetric $\alpha$-stable CARMA processes in \Cref{Sec2} and present some basic facts on Whittle estimation for  CARMA processes with finite second moments  in \Cref{sec:Section:3} which motivates our approach. Then, in  \Cref{Sec4}, we show the convergence of the Whittle function
for symmetric $\alpha$-stable CARMA processes and deduce that the Whittle estimator is not consistent for general symmetric $\alpha$-stable CARMA processes. However, we
show that the Whittle estimator  is a consistent estimator for symmetric $\alpha$-stable CAR(1) processes.
Finally, in \Cref{Sec5}, we demonstrate the theoretical results through a simulation study where we compare the performance of our estimator with the estimator proposed in \cite{Garcia:Klu:Mueller:2011}.
The behaviour of the Whittle estimator and the behaviour of the estimator of \cite{Garcia:Klu:Mueller:2011} is very similar in our simulation study.
In particular, in the simulation setup of  \cite{Garcia:Klu:Mueller:2011} the Whittle estimator performs excellent and leads to the presumption
that the Whittle estimator is converging.
Conclusions are given in \Cref{Conclusion}.
 Some auxiliary results on the asymptotic behaviour of the sample autocovariance function of symmetric $\alpha$-stable CARMA processes are postponed to the Appendix.
	
	%In order to reach an appropriate adjustment, we first review some basics concerning the light-tailed CARMA processes.  In the second section, we then introduce the heavy-tailed setting. Subsequently, we state our main result, namely, that the adjusted Whittle function converges in probability to random function which implies that the Whittle estimator does not converge to the true value in probability. Therefore, we see that
	%the adjusted Whittle estimator is not suited for the estimation of the parameters of an equidistantly sampled heavy-tailed CARMA process. To tackle the question whether the
	%other methods which are known to be useful in the light-tailed setting work better, we eventually review them. Unsurprisingly, they are not suited as well. Therefore, the problem of estimating the parameters of a heavy-tailed equidistantly sampled CARMA process remains open.

	\subsection*{Notation}
% We write $\vecc(A)$ for the vectorization of $A$ and $A\otimes B$ for the Kronecker product of $A$ and $B$ where $B$ is any matrix.
%The $n$-dimensional identity matrix is denoted as $I_n$. %For the gradient operator we use the notation $\nabla$ and we denote consequently $\nabla_\vartheta g(\vartheta_0)$ as shorthand for $[\frac{\partial}{\partial \vartheta_1}g(\vartheta),\frac{\partial}{\partial \vartheta_2}g(\vartheta),\ldots]|_{{\vartheta=\vartheta_0}}$ for any vector-valued function $g$. We furthermore interpret the gradient of a scalar function $g$ as a row vector and $\nabla^2_\vartheta g(\cdot)$ as its Hessian matrix. If the context allows, we write $\nabla_\vartheta g(\vartheta_0):=\nabla_\vartheta\vecc(g(\vartheta_0))$ for some matrix-valued function $g$.
We write $\mathcal{B}(A)$ for  the Borel-$\sigma$-algebra of a set $A$. The space
 $L^p(A)$ for $A\in\mathcal{B}(A)$ and $0<p<\infty$ denotes the set of measurable functions $f:A\to\R$ which satisfy $\int_A|f(t)|^p\,dt<\infty$.
Furthermore, for some function $f:A\to\R$ we write $f^+$ for the positive part $f^+(t)=\max\{0,f(t)\}$ and $f^-$ for the negative part $f^-(t)=\max\{0,-f(t)\}$.
The function $\Gamma$ describes the Gamma function. The indicator function of a set $A$ is denoted by $\1_A$ and the signum function by $\sign(z)=\1_{\{z>0\}}-\1_{\{z<0\}}$.
For the real and imaginary part of a complex valued $z$ we use the notation $\Re(z)$ and $\Im(z)$, respectively.
The $N$-dimensional identity matrix is denoted as $I_N$.	For some matrix $A$, we write
 $A^\top$ for its transpose. For the gradient operator we use the notation $\nabla$ and we write consequently $\nabla_\vartheta g(\vartheta_0)$ as shorthand for $[\frac{\partial}{\partial \vartheta_1}g(\vartheta),\frac{\partial}{\partial \vartheta_2}g(\vartheta),\ldots]^{\top}|_{{\vartheta=\vartheta_0}}$ for any vector-valued function $g$.
The symbols $\overset{a.s.}{\longrightarrow}$, $\overset{\P}{\longrightarrow}$ and $\overset{\mathcal D}{\longrightarrow}$
denote almost sure convergence, convergence in probability and convergence in distribution, respectively.
 %Throughout the article, $\| \cdot\|$ denotes an arbitrary sub-multiplicative matrix norm.

\section{Symmetric $\alpha$-stable CARMA processes}\label{Sec2}

In this paper, we consider a CARMA process driven by a symmetric $\al$-stable L\'evy process.

\begin{definition}
A random variable $Z$ is called $\alpha$-stable distributed, $\alpha\in \left(0,2\right]$, if $Z$ has the characteristic function $\E(\e^{izZ})=\exp(\varphi_Z(z))$
with
$$\varphi_Z(z)=\begin{cases}
    -\sigma^{\alpha}|z|^{\alpha}(1-i\beta(\operatorname{sign}(z))\tan\left(\pi \frac{\alpha}{2}\right)+i\mu z, &\text{for $\alpha \neq 1$,}\\
    -\sigma|z|(1+i\beta(\operatorname{sign}(z))\log|z|\left(\frac{2}{\pi}\right)+i\mu z,  &\text{for $\alpha=1$,}\\
\end{cases}
\qquad z\in\R,$$
and $\beta \in [-1,1],\sigma>0$ and $\mu \in \R$.  The parameter $\alpha$ is the index of stability.
We write $Z\sim S_{\alpha}(\sigma,\beta,\mu)$.
\end{definition}
In the case $\beta=0$ and $\mu=0$, the random variable $Z$ is symmetric. For $\alpha=2$ we get a normally distributed random variable.
Furthermore, for $\alpha\in(0,2)$ not all moments of $Z$ exist. To be more precise
\begin{align}
			\E|Z|^p<\infty \quad \text{ for } 0<p<\alpha \quad\text{ and } \quad \E|Z|^p=\infty \quad \text{ for } p\geq \alpha.
\end{align}
Important for us is likewise that in the symmetric case
\beao
    \lim_{n\to\infty} n\P(|Z|>n^{1/\alpha})=C_{\alpha}{\sigma ^\alpha}
\eeao
with
\beam \label{C_alpha}
    C_{\alpha}=\left\{\begin{array}{cl}
        \frac{1-\alpha}{\Gamma(2-\alpha)\cos(\pi\alpha/2)},  & \text{ if } \alpha\not=1\\
        \frac{2}{\pi}, & \text{ if } \alpha=1
    \end{array}\right.,
\eeam
see Property 1.2.15 of \cite{Samrodonitsky:Taqqu:1994} where
more details on stable distributions can be found as well.

A one-sided Lévy process $(L_t)_{t \geq 0}$ is a stochastic process with stationary and independent increments satisfying $L_0=0$ almost surely
and having continuous in probability sample paths; see the book of \cite{Sato} on Lévy processes.
A two-sided Lévy process $L=(L_t)_{t\in \R}$ can be constructed from two independent one-sided Lévy processes $(L^{[1]}_t)_{t\geq 0}$ and $(L^{[2]}_t)_{t\geq 0}$
through
$$L_t=L_{t}^{[1]}\mathbf{1}_{ \{t\geq 0\}}-\lim_{s\uparrow -t}L^{[2]}_{s}\mathbf{1}_{\{t<0\}}, \quad t \in \R.$$
Then, an  $\alpha$-stable Lévy process $L^{(\alpha)}=(L_t^{(\alpha)})_{t\in\R}$ is a Lévy process with
increments
$$L_t^{(\alpha)}-L_s^{(\alpha)}\sim S_\alpha(\sigma(t-s)^{1/\alpha},\beta,\mu),\quad s<t,$$
and it is symmetric, if additionally $\beta=\mu=0$.

In the following, we consider  a parametric family of stationary  CARMA processes. Let \linebreak $\Theta\subseteq \R^{N(\Theta)}$, $N(\Theta)\in\N$, be a parameter space, $p\in\N$ be fixed and for any
$\vartheta\in\Theta$ let  $a_1(\vartheta),\,\ldots,\,a_p(\vartheta),$ \linebreak $\,c_0(\vartheta),\,\ldots,\,c_{p-1}(\vartheta)\in\R,\,a_p(\vartheta)\neq 0$ and   $c_j(\vartheta)\not=0$ for some $j\in\{0,\ldots,p-1\}$.
Furthermore, define
\begin{eqnarray*} \label{A}
 A(\vartheta)&:=&
 \begin{pmatrix}
  0      & 1            & 0            & \ldots & 0            \\
  0      & 0            & 1            & \ddots & \vdots       \\
  \vdots & \vdots       & \ddots       & \ddots & 0            \\
  0      & 0            & \ldots       & 0      & 1            \\
  -a_p(\vartheta)   & -a_{p-1}(\vartheta)     & \ldots       & \ldots & -a_1(\vartheta)
 \end{pmatrix}\in \R^{p\times p},
 \end{eqnarray*}
 \begin{eqnarray*}
 q(\vartheta)&=&\inf\{j\in\{0,\ldots,p-1\}:\,c_l(\vartheta)=0\,\,\forall\, l>j\} \quad \text{ with }\quad \inf\emptyset:=p-1,\\
 c(\vartheta)&:=&(c_{q(\vartheta)}(\vartheta),\,c_{q(\vartheta)-1}(\vartheta),\,\ldots,\,c_{0}(\vartheta),0,\ldots,0)^\top\in\R^p.
\end{eqnarray*}
Throughout the paper we assume that the eigenvalues of $A(\vartheta)$ have strictly negative real parts for $\vartheta \in \Theta$.
%Due to \citet{Brockwell2001} (cf. \citet{Zadrozny1988} for possible non-stationary Gaussian CARMA processes in the econometric literature)
The family of   CARMA processes $Y(\vartheta)=(Y_t(\vartheta))_{t\geq 0}$ is then defined  via the controller canonical state space representation: Let $(L_t)_{t\geq 0}$
be a  Lévy process  and
let $(X_t(\vartheta))_{t\geq 0}$ be a strictly stationary solution to the stochastic differential equation
\begin{subequations} \label{Equation 1.1} \label{2.1}
 \begin{equation} \label{Equation 1.1a}
  {\rm d}X_t(\vartheta) = A(\vartheta) X_t(\vartheta)\,{\rm d}t + e_p\,{\rm d}L_t,\quad t\geq 0,
 \end{equation}
 where $e_p$ denotes the $p$-th unit vector in $\R^p$. Then the process
 \begin{equation} \label{Equation 1.1b}
  Y_t(\vartheta) := c(\vartheta)^\top X_t(\vartheta),\quad t\geq 0,
 \end{equation}
\end{subequations}
 is called \textit{CARMA process} of order $(p,\,q(\vartheta))$. If the Lévy process is $\alpha$-stable then $Y(\vartheta)$ is an \textit{$\alpha$-stable CARMA process}.
 Necessary and sufficient conditions for the existence of strictly stationary CARMA processes are given in \cite{Brockwell:Lindner:2009}.
  In particular, the assumptions are satisfied for $\alpha$-stable Lévy processes and Lévy processes with finite second moments.
  A CARMA process can be interpreted as the stationary solution to the formal $p$-th order stochastic differential equation
\[
 a_\vartheta({\rm D})Y_t(\vartheta)=c_\vartheta({\rm D}){\rm D}L_t,\quad t\geq 0,
\]
where ${\rm D}$ denotes the differential operator with respect to $t$ and
\[
 a_\vartheta(z):= z^p+a_1(\vartheta)z^{p-1}+\ldots+a_p(\vartheta)\qquad\text{and}\qquad c_\vartheta(z):=c_0(\vartheta)z^q+c_1(\vartheta)z^{q-1}+\ldots+c_{q(\vartheta)}(\vartheta)
\]
are the autoregressive and the moving average polynomials, respectively. Hence, CARMA processes can be seen
 as the continuous-time analogue of discrete-time ARMA processes.
Due to \eqref{Equation 1.1} the CARMA process $Y(\vartheta)$ also has the moving average (MA) representation
\beam \label{kernel}
    Y_t(\vartheta)=\int_{-\infty}^{\infty}g_\vartheta(t-s)\,dL_s^{(\alpha)}, \quad t\geq 0, \quad \text{ with }
        \quad g_\vartheta(t)=c(\vartheta)^\top\e^{A(\vartheta) t}e_p\1_{\left[0,\infty\right)}(t),
\eeam
which plays an important role in this paper.

\section{Whittle estimation for CARMA processes with finite second moments} \label{sec:Section:3}
%%%%%%%%%%%%%%%%%%%%%%%%%%%%%%%%%%%%%%%%%%%%%%%%%%%%%%%%%%%%%%%%%%%%%%%%%%%%%%%%%%%%%%%%%%%%%%%%%%%%%%%%%%%%%%%%
%%%%%%%%%%%%%%%%%%%%%%%%%%%%%%%%%%%%%%%%%%%%%%%%%%%%%%%%%%%%%%%%%%%%%%%%%%%%%%%%%%%%%%%%%%%%%%%%%%%%%%%%%%

Before we present the results for Whittle estimation of discrete-time sampled symmetric $\alpha$-stable CARMA processes
$Y=(Y_t)_{t\geq 0}:=(Y_t(\vartheta_0))_{t\geq 0}:=Y(\vartheta_0)$,
$\vartheta_0\in\Theta$, with observations  $Y_\Delta,\ldots,Y_{n\Delta}$ ($\Delta>0$ fixed), we repeat
the basic results on Whittle estimation of CARMA processes with finite second moments which motivates our approach.
This includes, in particular, Brownian motion driven CARMA processes which are $2$-stable.
For a CARMA process $Y(\vartheta)$ driven by a Lévy process $L$ with $\E(L_1)=0$ and $0<\E(L_1^2)=\sigma_L^2<\infty$ the spectral density is
\beao
    f_Y(\omega,\vartheta)=\frac{\sigma_L^2}{2\pi}\frac{|c_\vartheta(i\omega)|^2}{|a_\vartheta(i\omega)|^2}, \quad \omega \in\R,
\eeao
and for the discrete-time sampled process $Y^{(\Delta)}(\vartheta):=(Y_{k\Delta}(\vartheta))_{k\in\N}$ it is
\beao
    f_Y^{(\Delta)}(\omega,\vartheta)&=&\frac{\sigma_L^2}{2\pi}\int_0^\Delta\left|\sum_{j=-\infty}^\infty g_\vartheta(u+j\Delta)\e^{i\omega j}\right|^2\,du\\
        &=&\frac{\sigma_L^2}{2\pi}\int_0^{\Delta}\left|c(\vartheta)^\top\e^{A(\vartheta)u}(I_p-\e^{A(\vartheta)\Delta+i\omega I_p})^{-1}e_p\right|^2\,du, \quad \omega\in[-\pi,\pi],
\eeao
see \cite{fasen2013a}, Example 2.4.
The empirical version of the spectral density is
	the periodogram $I_n:[-\pi,\pi]\rightarrow \R$ defined as
	\begin{eqnarray}\label{Per_ACVF}
	I_n(\omega)=\frac{1}{2\pi n}\left|\sum_{k=1}^{n}Y_{k}^{(\Delta)}e^{ik\omega}\right|^2
	=\frac{1}{2\pi}\sum_{h=-n+1}^{n-1}\overline{\gamma}_n(h)e^{-ih\omega}, \quad \omega \in [-\pi, \pi],
    \end{eqnarray}	
    where 	
	\beam \label{ACF}
        \overline{\gamma}_n(h):=\frac1n\sum_{k=1}^{n-h}Y^{(\Delta)}_{k+h}Y^{(\Delta)}_{k}=:\overline{\gamma}_n(-h)\quad \text{ for }h=0,\ldots,n-1,
    \eeam
    is the empirical autocovariance function. The Whittle estimator is based on the distance between the spectral density and the periodogram.
	To be more precise, the  Whittle estimator $\widehat\vartheta_n$ is the minimizer of the  Whittle function
    $$W_n(\vartheta):=\frac{1}{2n}\sum_{j=-n+1}^{n}\Big(f_Y^{(\Delta)}(\omega_j, \vartheta)^{-1}I_n(\omega_j)+\log \left(f_Y^{(\Delta)}(\omega_j,\vartheta)\right)\Big), \quad \vartheta\in\Theta,$$
	with
	$\omega_j=\frac{\pi j}{n}$ for $ j=-n+1,\ldots,n.$
    \cite{Fasen:Mayer:2020}, Theorem 1, show that under very general assumptions the Whittle estimator is  consistent
    and asymptotically normally distributed.

Since the spectral density depends on the variance $\sigma_L^2$ of the driving Lévy process, \cite{Fasen:Mayer:2020} present an adjusted version
of the Whittle estimator which is independent of the variance parameter $\sigma_L^2$. However, to derive this estimator
some background on the Kalman filter is required.

 The Kalman filter is going back to \cite{kalman1960}
and is well investigated in control theory (see  \cite{Chui:Chen:2009}) and time series analysis  (see \cite{Brockwell:Davis:1991});
in the context of CARMA processes we refer to \cite{QMLE}.
Therefore, note that the Riccati equation
\begin{eqnarray*}
\Omega^{(\Delta)}(\vartheta)&=&e^{A(\vartheta)\Delta}\Omega^{(\Delta)}(\vartheta) \left(e^{A(\vartheta)\Delta}\right)^\top + \int_0^{\Delta}e^{A(\vartheta)u}e_pe_p^\top \e^{A(\vartheta)^\top u}du\\&&- \left( e^{A(\vartheta)\Delta}\Omega^{(\Delta)} (\vartheta)c(\vartheta)\right)\left( c(\vartheta)^\top\Omega^{(\Delta)}(\vartheta) c(\vartheta)\right)^{-1}\left( e^{A(\vartheta)\Delta}\Omega^{(\Delta)}(\vartheta) c(\vartheta)\right)^\top
\end{eqnarray*} has an unique positive semidefinite solution $\Omega^{(\Delta)}(\vartheta)$
such that the Kalman gain matrix
$$K^{(\Delta)}(\vartheta)=\left( e^{A(\vartheta)\Delta}\Omega^{(\Delta)}(\vartheta) c(\vartheta)\right)\left( c(\vartheta)^\top\Omega^{(\Delta)}(\vartheta) c(\vartheta)\right)^{-1}$$
is well-defined.
Define the polynomial $\Pi^{(\Delta)}$  as
\begin{eqnarray*} \label{Pi:rep}	
\Pi^{(\Delta)}(z,\vartheta):=\left(1 -c(\vartheta)^\top\left(I_{N}-(e^{A(\vartheta)\Delta}-K^{(\Delta)}(\vartheta)c(\vartheta)^\top)z\right)^{-1}K^{(\Delta)}(\vartheta)z\right), \quad z\in\C,
\end{eqnarray*}
and its inverse
\begin{eqnarray*}
    \Pi^{(\Delta)}(z,\vartheta)^{-1}:=1+c(\vartheta)^\top\sum_{j=1}^{\infty}\left(e^{A(\vartheta)\Delta}\right)^{j-1}K^{(\Delta)}(\vartheta)z^j, \quad z\in\C,
\end{eqnarray*}
which gives $\Pi^{(\Delta)}(z,\vartheta)^{-1}\Pi^{(\Delta)}(z,\vartheta)=z$.
If the driving Lévy process $L$ has finite second moments and mean zero, then
\begin{eqnarray} \label{linearInnovations}
\varepsilon_{k}^{(\Delta)}(\vartheta)=\Pi^{(\Delta)}(\mathsf{ B},\vartheta)Y^{(\Delta)}_k(\vartheta), \quad k \in \N,			%&=&Y_k^{(\Delta)}(\vartheta)-\sum_{j=1}^{\infty}c(\vartheta)(e^{A(\vartheta)\Delta}-K^{(\Delta)}(\vartheta)c(\vartheta))^{j-1}
%K^{(\Delta)}(\vartheta)Y_{k-j}^{(\Delta)}(\vartheta).
\end{eqnarray}
is a sequence of uncorrelated identically distributed random variables. They are the noise of the
 ARMA representation of the discretely sampled CARMA process $Y^{(\Delta)}(\vartheta)$ (cf.  \cite{QMLE}).
Define
\beam \label{variance:linearinnovations}
    \sigma^2_{\varepsilon^{(\Delta)}}(\vartheta):=\sigma_L^2c(\vartheta)^{\top}\Omega^{(\Delta)}(\vartheta)c(\vartheta).
\eeam
Then
$\E(\varepsilon_{k}^{(\Delta)}(\vartheta)^2)= \sigma^2_{\varepsilon^{(\Delta)}}(\vartheta)$ is the variance of the white noise.
By an application of  \cite{Brockwell:Davis:1991}, Theorem 11.8.3, the spectral density of $Y^{(\Delta)}(\vartheta)$ has
the representation
    \begin{eqnarray}
    	\label{specY}
		\f(\omega,\vartheta)=\frac{|\Pi^{(\Delta)}(e^{i\omega},\vartheta)^{-1}|^2 \sigma^2_{\varepsilon^{(\Delta)}}(\vartheta)}{2\pi},
            \quad  \omega \in [-\pi,\pi].
	\end{eqnarray}
Note that $\Pi^{(\Delta)}$ and hence, $\Pi^{(\Delta)\,-1}$ do not depend on the variance $\sigma_L^2$ of
the driving Lévy process.
This motivates the definition of the (adjusted) Whittle function
\begin{align*}W^{(A)}_n(\vartheta)&:=\frac{\pi}{n}\sum_{j=-n+1}^{n}|\Pi^{(\Delta)}(e^{i\omega_j},\vartheta)|^{2}I_n(\omega_j)
\end{align*}
and the (adjusted) Whittle estimator
\beam \label{adjusted Whittle}
    \widehat\vartheta^{(\Delta,A)}_n:=\arg\min_{\vartheta \in \Theta}W^{(A)}_n(\vartheta).
\eeam
The adjusted Whittle estimator has desirable properties in the light-tailed setting as derived
in \cite{Fasen:Mayer:2020}.
To present these results we need some further assumptions.

\setcounter{assumptionletter}{0}
\begin{assumptionletter}\strut
         \label{as_D}
         \renewcommand{\theenumi}{(A\arabic{enumi})}
         \renewcommand{\labelenumi}{\theenumi}
         \begin{enumerate}
                 \item \label{as_D1} The parameter space $\Theta$ is a compact subset of $\R^{N(\Theta)}$.
                 \item \label{as_D8} The true parameter $\vartheta_0$ is an element of the interior of $\Theta$.
             %    \item \label{as_D2}  $L$ is a symmetric $\alpha$-stable Lévy process with $\alpha\in(0,2)$.
                 \item \label{as_D3} The eigenvalues of $A(\vartheta)$ have strictly negative real parts for $\vartheta \in \Theta$.
                 %\item \label{as_D4} The functions $\vartheta \mapsto A_\vartheta$ and  $\vartheta \mapsto c_\vartheta$ are three times continuously differentiable.
                 \item \label{as_D5} For all $\vartheta \in \Theta$ the zeros of $c_\vartheta(z)=c_0(\vartheta)z^{q(\vartheta)}+c_1(\vartheta)z^{q(\vartheta)-1}+\ldots +c_{q(\vartheta)}$ are different from the eigenvalues of $A(\vartheta)$.
                 \item \label{as_D6} For any $\vartheta,\vartheta'\in\Theta$ we have $(c(\vartheta),A(\vartheta))\not=(c({\vartheta'}),A({\vartheta'}))$. %The family of CARMA processes $(Y_\vartheta(t))_{t\in\R}$ is identifiable from the spectral density.
                 \item \label{as_D7} For all $\vartheta \in \Theta$ the spectrum of $A(\vartheta)$ is a subset of $\{ z \in \C: -\frac{\pi}{\Delta} < \mathfrak I(z) < \frac{\pi}{\Delta} \}$.
                 \item  For any $\vartheta_1,\ \vartheta_2 \in \Theta$, $\vartheta_1\neq \vartheta_2$, there exists some $z\in\C$ with $|z|=1$ and $\Pi^{(\Delta)}(z,\vartheta_1)\neq \Pi^{(\Delta)}(z,\vartheta_2)$.
                 \item \label{as_D9} The maps $\vartheta\mapsto A(\vartheta)$ and $\vartheta\mapsto c(\vartheta)$ are three times continuously differentiable.
          %       \item \label{as_D10} For every $\epsilon > 0$ there exists a
%$\delta(\epsilon) > 0$ such that
%\begin{equation*}
 %     $$\mathscr{Q}(\vartheta^\ast) \leq \min_{\vartheta \in B_\epsilon(\vartheta^*)^c \cap \Theta}   \mathscr{Q}(\vartheta)    - \delta(\epsilon),$$
%\end{equation*}
%where $B_\epsilon(\vartheta^*)$ is the open ball with center $\vartheta^*$ and radius $\epsilon$.
% \item \label{as_D11} The Fisher information matrix  of the quasi maximum likelihood estimator is non-singular.
         \end{enumerate}
\end{assumptionletter}
Throughout the paper we suppose that Assumption A holds.

\begin{theorem}[\cite{Fasen:Mayer:2020}, Theorem 1 and Theorem 4]  \label{Theorem 3.1}
Let $L$ be a Lévy process with $\E(L_1)=0$ and $\E(L_1^2)<\infty$.
\begin{enumerate}
     \item Then, $\widehat\vartheta^{(\Delta,A)}_n\overset{a.s.}{\longrightarrow}\vartheta_0$ as $n\to\infty$.
    \item Suppose further $\E(L_1^4)<\infty$ and for any $c \in \C^{N(\Theta)}$ there exists an $\omega^* \in [-\pi,\pi]$ such that $\nabla_\vartheta |\Pi^{(\Delta)}(e^{i\omega^*},\vartheta_0)|^{-2}c \neq 0_{N(\Theta)}$. Then, there exists a positive definite matrix $\Sigma_{W^{(A)}}\in\R^{N(\Theta)\times N(\Theta)}$
          such that
          as $n\to \infty$, $$\sqrt{n}\left(\widehat{\vartheta}_n^{(\Delta,A)}-\vartheta_0\right)\overset{\mathcal D}{\longrightarrow}\mathcal{N} (0,\Sigma_{W^{(A)}}).$$
          If $L$ is a Brownian motion, then
          \beao
            \Sigma_{W^{(A)}}=4\pi \left[\int_{-\pi}^{\pi}\nabla_\vartheta \log\left(|\Pi^{(\Delta)}(e^{i\omega},\vartheta_0)|^{-2}\right)^\top\nabla_\vartheta \log\left(|\Pi^{(\Delta)}(e^{i\omega},\vartheta_0)|^{-2}\right) d\omega\right]^{-1}.
          \eeao
\end{enumerate}
\end{theorem}

\section{Whittle estimation for symmetric $\alpha$-stable CARMA processes}\label{Sec4}
The topic of this paper is Whittle estimation of the parameters of symmetric $\alpha$-stable CARMA processes.
There are several arguments for the conjecture that the Whittle estimator might converge:
\begin{itemize}
    \item[-] For light-tailed CARMA processes the Whittle estimator is consistent and asymptotically normally distributed (see \Cref{Theorem 3.1}).
    \item[-] In several simulations  for symmetric $\alpha$-stable CARMA processes as, e.g.,
    in the setup of \cite{Garcia:Klu:Mueller:2011} (see \Cref{Sec5}), it seems that the Whittle estimator converges to the true parameter.
    \item[-] In the context of ARMA processes the ideas of Whittle estimation for ARMA processes with finite second moments
        could be transferred to ARMA processes with infinite second moments (see \cite{Mikosch:Gardrich:Klueppelberg:1995}). In particular, equidistant sampled CARMA
        processes with finite second moments have a weak ARMA representation. %However, a challenge is that the noise in this ARMA representation depends on
        %the model parameters as well.
\end{itemize}
%The consistency results for symmetric $\alpha$-stable ARMA processes in \cite{Mikosch:Gardrich:Klueppelberg:1995} and for light-tailed
%CARMA processes (see \Cref{Theorem 3.1}) as well some simulation results in \Cref{Sec5} give the impression that the Whittle estimator for heavy-tailed
%CARMA processes is converging as well. However, there are also some simulation results in \Cref{Sec5} which are not so clear.
The aim of this paper
is to mathematically justify why the Whittle estimator for symmetric $\alpha$-stable CARMA processes is in general not converging.
 Therefore, we assume for the rest of the paper that the driving process $L$ of the CARMA process is a symmetric $\alpha$-stable L\'evy process $L^{(\alpha)}$ with $L^{(\alpha)}_1\sim S_\alpha(\sigma,0,0)$ for some $\sigma>0,\ \alpha \in (0,2)$ and that $Y$ is a symmetric $\alpha$-stable CARMA process with kernel function $g(t)=c^\top\e^{A t}e_p\1_{\left[0,\infty\right)}(t)$ as given in \eqref{kernel}.
In analogy to \eqref{adjusted Whittle} for light-tailed CARMA processes,
for symmetric $\alpha$-stable CARMA processes the (adjusted) \textsl{ Whittle function} is
\begin{align*}W^{(\alpha)}_n(\vartheta)&:=\frac{\pi}{n^{2/\alpha}}\sum_{j=-n+1}^{n}|\Pi^{(\Delta)}(e^{i\omega_j},\vartheta)|^{2}I_n(\omega_j), \quad \vartheta\in\Theta,
\end{align*}
and the (adjusted) \textsl{ Whittle estimator} is
$$\widehat\vartheta^{(\Delta,\alpha)}_n:=\arg\min_{\vartheta \in \Theta}W^{(\alpha)}_n(\vartheta).$$

	\begin{theorem}\label{consistency3}
Suppose Assumption A holds. Define $$W^{(\alpha)}(\vartheta):=\frac 1 {2\pi}\int_0^{\Delta}\left[ \int_{-\pi}^{\pi}|\Pi^{(\Delta)}(e^{i\omega},\vartheta)|^2\left|\sum_{j=-\infty}^{\infty} g(\Delta j-s)e^{-ij\omega}\right|^2 d\omega \right]dL_s^{({\alpha/2})},$$ 	where
$L^{(\alpha/ 2)}=(L^{(\alpha/ 2)}_t)_{t\geq 0}$ is an ${\alpha}/{2}$-stable Lévy process with
$L^{(\alpha/2)}_1\sim S_{\alpha/2}({\sigma^2}\left(C_\alpha/C_{\alpha/2}\right)^{2/\alpha},1,0)$ and the constants $C_\alpha$ and $C_{\alpha/2}$ are defined as in \eqref{C_alpha}.	
Then, as $n\to\infty$,
$$ (W^{(\alpha)}_n(\vartheta))_{\vartheta\in\Theta}\overset{\mathcal D}{\longrightarrow} (W^{(\alpha)}(\vartheta))_{\vartheta\in\Theta} \quad \text{ in } \quad (\mathcal{C}(\Theta),\|\cdot\|_{\infty}),$$
where $\mathcal C(\Theta)$ is the space of continuous functions on $\Theta$ with the supremum norm $\|\cdot\|_{\infty}$.
	\end{theorem}

\begin{proof}
We approximate $|\Pi^{(\Delta)}(e^{i\omega_j}, \vartheta)|^{2}$ by the Cesàro sum of its Fourier series of size $M$  for $M$ sufficiently large. Define \begin{align*}
	q_M(\omega,\vartheta)&:=\frac{1}{M}\sum_{j=0}^{M-1}\left(\sum_{|k|\leq j}{b}_k(\vartheta) e^{-ik\omega} \right)=\sum_{|k|<M}\left(1-\frac{|k|}{M}\right)b_k(\vartheta) e^{-ik\omega} \quad \text{ with }\\b_k(\vartheta)&:= \frac 1 {2\pi}\int_{-\pi}^{\pi} |\Pi^{(\Delta)}(e^{i\omega}, \vartheta)|^{2}e^{ik\omega}d\omega
	\end{align*}
	and $$W_{M,n}^{(\alpha)}(\vartheta):=\frac{\pi}{n^{2/\alpha}}\sum_{j=-n+1}^{n}q_M(\omega_j,\vartheta)I_n(\omega_j), \quad \vartheta\in\Theta.$$
	Let $\varepsilon_1>0$. A conclusion of Lemma 6 of \cite{Fasen:Mayer:2020} is that there exists an
$M_0(\varepsilon_1)\in \N$ such that for $M\geq M_0(\varepsilon_1)$
\begin{eqnarray}\label{q_M_dist}\sup_{\omega \in [-\pi,\pi]}\sup_{\vartheta\in \Theta}|q_M(\omega,\vartheta)-|\Pi^{(\Delta)}(e^{i\omega},\vartheta)|^2|<\varepsilon_1.
\end{eqnarray} Similar arguments as in the proof of Proposition 2 in \cite{Fasen:Mayer:2020} yield
$$\sup_{\vartheta \in \Theta }\left|W_n^{(\alpha)}(\vartheta)-W_{M,n}^{(\alpha)}(\vartheta) \right|\leq\frac{ \varepsilon_1}{n^{2/\alpha -1}}\overline \gamma_n(0) \quad \text{ for } M\geq M_0(\epsilon_1).$$
	Due to \Cref{CorDrapatz} $$\frac{ 1}{n^{2/\alpha -1}}\overline \gamma_n(0)\overset{\mathcal D}{\longrightarrow}\int_0^\Delta \sum_{j=-\infty}^{\infty}g(\Delta j-s)^2 dL^{({\alpha/2})}_s \quad \text{ as } n\to\infty.$$
%Thus, the sequence of probability measures $\left(\P^{n^{-2/\alpha +1}\overline \gamma_n(0)}\right)_{n\in \N}$ is tight.
%	Hence, for any $\epsilon,\delta>0$, there exists an $N_0(\delta)\in\N$ and an $M_0(\epsilon,\delta)\in\N$ such that
%\begin{align*}
%\P\left(\sup_{\vartheta \in \Theta }\left|W_n^{(\alpha)}(\vartheta)-W_{M,n}^{(\alpha)}(\vartheta) \right|>\varepsilon\right)\leq  \delta
%\quad \text{ for }	\quad n\geq N_0(\delta),\ M\geq M_0(\epsilon,\delta).
%\end{align*}
Therefore, we have for any $\epsilon_2>0$
	\begin{eqnarray}\label{T1}
\lim_{M\to \infty}\limsup_{n\to \infty} \P\left(\sup_{\vartheta \in \Theta }\left|W_n^{(\alpha)}(\vartheta)-W_{M,n}^{(\alpha)}(\vartheta) \right|>\varepsilon_2\right)=0.
	\end{eqnarray}
Furthermore, representation \eqref{Per_ACVF} gives
	\begin{eqnarray} \label{F.1}
	\nonumber W_{M,n}^{(\alpha)}(\vartheta) &=&\sum_{|k|<M}\left(\left(1-\frac{|k|}{M}\right)b_k(\vartheta)\left(n^{-2/\alpha+1}\sum_{|h|<n}\overline\gamma_n(h)\right){\left(\frac{1}{2n}\sum_{j=-n+1}^{n}e^{-i(k+h)\omega_j}\right)}\right)\\
&=&\sum_{|k|<M}\left(1-\frac{|k|}{M}\right)b_k(\vartheta)n^{-2/\alpha+1}\overline\gamma_n(-k).	\end{eqnarray}
	We define
\beam \label{F.2}
		W^{(\alpha)}_M(\vartheta):=\sum_{|k|<M}\left(1-\frac{|k|}{M}\right)b_{k}(\vartheta)\int_0^\Delta \sum_{j=-\infty}^{\infty}g(\Delta(j+k)-s)g(\Delta j-s)dL_s^{(\alpha/2)}.
\eeam
\begin{comment}
 Applying the continuous mapping theorem and \Cref{CorDrapatz}$$ \sum_{s=1}^{r}c_s W_{M,n}^{(\alpha)}(\vartheta_s)=\sum_{s=1}^{r}\sum_{|k|<M}\left(1-\frac{|k|}{M}\right)c_sb_k(\vartheta_s)n^{-2/\alpha+1}\overline\gamma_n(-k)\\ \overset{\mathcal D}{\longrightarrow}
	\sum_{s=1}^{r}c_s  W_M^{(\alpha)}(\vartheta_s)$$
	for arbitrary $r\in \N$, $c_1,\ldots,c_r \in \R$ and $\vartheta_1,\ldots,\vartheta_r\in\Theta$. Thus, by Cram\' er-Wold, the finite dimensional distributions of $(W_{M,n}(\vartheta))_{\vartheta \in \Theta}$ converge to the finite dimensional distribution of $(W_{M}^{(\alpha)}(\vartheta))_{\vartheta \in \Theta}$.
\end{comment}	
Due to Assumption (A8) and the definition of $\Pi^{(\Delta)}$, there exists a constant $\mathfrak{C}>0$ such that for any $\delta>0$
 \begin{align*}
 \sup_{\substack{|\vartheta_1-\vartheta_2|<\delta\\ \vartheta_1,\vartheta_2\in\Theta,k \in \Z}} |b_k(\vartheta_1)-b_k(\vartheta_2)|&=\sup_{\substack{|\vartheta_1-\vartheta_2|<\delta\\ \vartheta_1,\vartheta_2\in\Theta,k \in \Z}} \left| \frac{1}{2\pi}\int_{-\pi}^{\pi} \left(|\Pi^{(\Delta)}(e^{i\omega}, \vartheta_1)|^{2}-|\Pi^{(\Delta)}(e^{i\omega},\vartheta_2) |^2\right)e^{ik\omega}d\omega \right|\\
	 &\leq \max_{|\vartheta_1-\vartheta_2|<\delta \atop \vartheta_1,\vartheta_2\in\Theta} \max_{\omega\in [-\pi,\pi]}||\Pi^{(\Delta)}(e^{i\omega},\vartheta_1) |^2-|\Pi^{(\Delta)}(e^{i\omega},\vartheta_2) |^2|\leq \mathfrak C\delta. \end{align*}
This means that $(b_k(\vartheta))_{\vartheta\in\Theta}$ is uniformly continuous.
By \Cref{CorDrapatz}, we have the joint convergence of the random vector $(\overline\gamma_n(-M+1),\ldots,\overline\gamma_n(M-1))$
implying with the representations \eqref{F.1}, \eqref{F.2} and  the continuous mapping theorem that
\beam
    (W_{M,n}^{(\alpha)}(\vartheta))_{\vartheta\in\Theta}\overset{\mathcal D}{\longrightarrow}(W_{M}^{(\alpha)}(\vartheta))_{\vartheta\in\Theta}\quad \text{ in } \quad (\mathcal{C}(\Theta),\|\cdot\|_{\infty}).
\eeam
Furthermore, \begin{align*}
W_M^{(\alpha)}(\vartheta)&=\int_0^\Delta\sum_{|k|<M}\left(1-\frac{|k|}{M}\right)b_{k}(\vartheta) \sum_{j=-\infty}^{\infty}g(\Delta(j+k)-s)g(\Delta j-s)dL_s^{(\alpha/2)}\\
&=\int_0^\Delta\sum_{|k|<M}\left(1-\frac{|k|}{M}\right)b_{k}(\vartheta) \sum_{j,\ell=-\infty}^{\infty}g(\Delta j-s)g(\Delta \ell-s)\left[\frac 1 {2\pi}\int_{-\pi}^{\pi}e^{-i(j+k-\ell)\omega}d\omega\right]dL_s^{(\alpha/2)}\\
&=\frac{1}{2\pi}\int_0^\Delta\left[\int_{-\pi}^{\pi}q_M(\omega,\vartheta)\left| \sum_{j=-\infty}^{\infty}g(\Delta j-s)e^{-ij\omega}\right|^2d\omega\right]dL_s^{(\alpha/2)}.
\end{align*}
By this, $$
W_M^{(\alpha)}(\vartheta)-W^{(\alpha)}(\vartheta)=\frac{1}{2\pi}\int_0^\Delta\left[\int_{-\pi}^{\pi}\left[q_M(\omega,\vartheta)-|\Pi^{(\Delta)}(e^{i\omega},\vartheta)|^2\right]\left| \sum_{j=-\infty}^{\infty}g(\Delta j-s)e^{-ij\omega}\right|^2d\omega\right]dL_s^{(\alpha/2)}$$ holds.
Since the process $L^{(\alpha/2)}$ is positive and increasing we obtain
\begin{eqnarray*}
\lefteqn{\sup_{\vartheta\in \Theta}|W_M^{(\alpha)}(\vartheta)-W^{(\alpha)}(\vartheta)|}\\&\leq& \frac{1}{2\pi} \int_0^\Delta\left[\int_{-\pi}^{\pi}\sup_{\vartheta \in \Theta}\left|q_M(\omega,\vartheta)-|\Pi^{(\Delta)}(e^{i\omega},\vartheta)|^2\right|\left| \sum_{j=-\infty}^{\infty}g(\Delta j-s)e^{-ij\omega}\right|^2d\omega\right]dL_s^{(\alpha/2)}
=:\widetilde W_M^{(\alpha/2)}.
\end{eqnarray*}
Note that by Property 3.2.2 of \cite{Samrodonitsky:Taqqu:1994},   $\widetilde W_M^{(\alpha/2)}\sim S_{\alpha/2}(\sigma_M,\beta_M,\mu_M)$ where $\beta_M=1,\ \mu_M=0$  and  $$\sigma_M^{\alpha/2}=\frac {\sigma^\alpha C_\alpha}{ C_{\alpha/2}}\int_0^\Delta  \left[\frac 1 {2\pi}\int_{-\pi}^{\pi}\sup_{\vartheta \in \Theta}\left|q_M(\omega,\vartheta)-|\Pi^{(\Delta)}(e^{i\omega},\vartheta)|^2\right|\left| \sum_{j=-\infty}^{\infty}g(\Delta j-s)e^{-ij\omega}\right|^2d\omega\right]^{\alpha/2}ds.$$
Due to Assumption \ref{as_D3} there exists a constant $\mathfrak{C}>0$ such that
$
	\left(\sum_{j=0}^{\infty}\left\| e^{A(\Delta j-s)}\right\|\right)^2 <\mathfrak C$ for any $s\in [0,\Delta]$. Thus,
\beao
    \int_0^\Delta\left[\int_{-\pi}^\pi \left|\sum_{j=-\infty}^{\infty}g(\Delta j-s)e^{-ij\omega}\right|^2 d\omega\right]^{\alpha/2}\hspace*{-0.4cm}ds&\leq & \int_0^\Delta\left[\int_{-\pi}^\pi \|c\|^2 \|e_p\|^2	\left(\sum_{j=0}^{\infty} \left\| e^{A(\Delta j-s)}\right\|\right)^2  d\omega\right]^{\alpha/2}\hspace*{-0.4cm}ds <\infty.
\eeao
A conclusion of this and \eqref{q_M_dist} is that $\sigma_M^{\alpha/2}\overset{M\to \infty}{\longrightarrow}0$ and hence,
 the characteristic function $\varphi_{\widetilde W_M^{\alpha/2}}$ of $\widetilde W_M^{(\alpha/2)}$ converges pointwise to $\varphi_{\widetilde W^{\alpha/2}}\equiv 1$.
  An application of L\'evys continuity theorem results then in $\widetilde W_M^{(\alpha/2)}\overset{\P}{\longrightarrow} 0$ as $M\to\infty$. Finally,
\beam
  \label{T3}  \sup_{\vartheta\in \Theta}|W_M^{(\alpha)}(\vartheta)-W^{(\alpha)}(\vartheta)|\overset{\P}{\longrightarrow} 0 \quad \text{ as } M\to\infty
\eeam
as well.
In view of \eqref{T1}-\eqref{T3}, Theorem 3.2 of \cite{Billingsley:1999} completes the proof.
\end{proof}

\begin{corollary} \label{Lemma:4.2}
Let the assumptions of \Cref{consistency3} hold. Define $G_{\vartheta,\vartheta_0}:[0,\Delta]\to \R$ as
$$G_{\vartheta,\vartheta_0}(u)=
\frac 1 {2\pi} \int_{-\pi}^{\pi}\left[|\Pi^{(\Delta)}(e^{i\omega},\vartheta)|^2-|\Pi^{(\Delta)}(e^{i\omega},\vartheta_0)|^2\right]\left|\sum_{j=-\infty}^{\infty} g(\Delta j-u)e^{-ij\omega}\right|^2 d\omega.$$
Then,
\begin{eqnarray*}
    W^{(\alpha)}(\vartheta)-W^{(\alpha)}(\vartheta_0)\sim S_{\alpha/2}(\sigma_{\vartheta,\vartheta_0},\beta_{\vartheta,\vartheta_0},0)
\end{eqnarray*}
is an ${\alpha}/{2}$-stable random variable with parameters
    \begin{eqnarray*}
    \beta_{\vartheta,\vartheta_0}&=&\frac{\int_{0}^{\Delta}(G_{\vartheta,\vartheta_0}^+(s))^{\alpha/2}-(G_{\vartheta,\vartheta_0}^-(s))^{\alpha/2}ds}{\int_0^{\Delta}|G_{\vartheta,\vartheta_0}(s)|^{\alpha/2}ds},\\
	\sigma_{\vartheta,\vartheta_0}^{\alpha/2}&=&\frac{\sigma^{\alpha}C_\alpha}{C_{\alpha/2}}\int_0^\Delta |G_{\vartheta,\vartheta_0}(s)|^{\alpha/2}ds.
		\end{eqnarray*}
%	and $C_{\alpha/2}$ as defined in \eqref{C_alpha}.
\end{corollary}

\subsection{Whittle estimation for symmetric $\alpha$-stable Ornstein-Uhlenbeck processes} \label{sec:4.1}

An Ornstein-Uhlenbeck process $Y_t(\vartheta)=\int_{-\infty}^{t}e^{\vartheta(t-s)}dL_s^{(\alpha)}$, $t\geq0$, sampled equidistantly has the AR(1) representation
\beao
    Y_k^{(\Delta)}(\vartheta)=\e^{\vartheta\Delta}Y_{k-1}^{(\Delta)}(\vartheta)+\xi_k^{(\Delta)}(\vartheta),
\eeao
where $\xi_k^{(\Delta)}(\vartheta)=\int_{(k-1)\Delta}^{k\Delta}\e^{\vartheta(k\Delta-s)}\,dL_s^{\alpha}$, $k\in\N$, is an iid symmetric $\alpha$-stable
sequence. Since the distribution of the white noise $\xi_k^{(\Delta)}(\vartheta)$ depends on $\vartheta$, the theory of \cite{Mikosch:Gardrich:Klueppelberg:1995} can not be applied
directly to estimate $\vartheta$ in this setting even though we have an AR$(1)$ representation. Thus, in this subsection we derive the
consistency of the Whittle estimator for symmetric $\alpha$-stable Ornstein-Uhlenbeck processes.

\begin{proposition}\label{OU}
Let $Y_t(\vartheta)=\int_{-\infty}^{t}e^{\vartheta(t-s)}dL_s^{(\alpha)}$, $t\geq0$, for $\vartheta \in \Theta \subseteq\left(-\infty,0\right)$ and $\Theta$ compact be a family of symmetric $\alpha$-stable Ornstein-Uhlenbeck processes. Then, as $n\to\infty$,
\beao
     W^{(\alpha)}_n(\vartheta)\overset{\mathcal D}{\longrightarrow}W_{OU}(\vartheta) S_{\alpha/2}^* \quad \text{ in } \quad (\mathcal{C}(\Theta),\|\cdot\|_{\infty}),
\eeao
where $S_{\alpha/2}^*$ is a positive $\alpha/2$-stable random variable and
\beao
    W_{OU}(\vartheta)=\frac{1}{2\pi}\int_{-\pi}^\pi|1-\e^{\vartheta\Delta+i\omega}|^2|1-\e^{\vartheta_0\Delta+i\omega}|^{-2}\,d\omega, \quad \vartheta \in \Theta.
\eeao
\end{proposition}
\begin{proof}
The Ornstein-Uhlenbeck process $Y(\vartheta)$ has the kernel function  $g_\vartheta(t)=e^{\vartheta t}\mathbf{1}_{[0,\infty)}(t)$ and the transfer function $ \Pi^{(\Delta)}(z,\vartheta)=1-e^{\vartheta \Delta}z.$ 	Therefore, an application of \Cref{consistency3} yields  as $n\to\infty$,
\begin{eqnarray*}
	W^{(\alpha)}_n(\vartheta)&\overset{\mathcal D}{\longrightarrow} &\frac 1 {2\pi}\int_0^{\Delta}\left[ \int_{-\pi}^{\pi}|\Pi^{(\Delta)}(e^{i\omega},\vartheta)|^2\left|\sum_{j=-\infty}^{\infty} g_{\vartheta_0}(\Delta j-s)e^{-ij\omega}\right|^2d\omega \right]dL_s^{{\alpha/2}}\\
	&=&\frac 1 {2\pi}\int_0^{\Delta}\left[ \int_{-\pi}^{\pi}|1-e^{\vartheta\Delta+i\omega}|^2\left|\sum_{j=1}^{\infty} \e^{\vartheta_0(\Delta j-s)}e^{-ij\omega}\right|^2d\omega \right]dL_s^{{\alpha/2}}\\
	&=&\frac{1}{2\pi}\int_{-\pi}^\pi|1-\e^{\vartheta\Delta+i\omega}|^2|1-\e^{\vartheta_0\Delta+i\omega}|^{-2}\,d\omega \int_0^\Delta e^{2\vartheta_0(\Delta-s)}dL_s^{(\alpha/2)}
	\end{eqnarray*}
in $(\mathcal{C}(\Theta),\|\cdot\|_{\infty})$.	
 Define $S_{\alpha/2}^*:=\int_0^\Delta e^{2\vartheta_0(\Delta-s)}dL_s^{(\alpha/2)}.$
	Due to Property 3.2.2 of \cite{Samrodonitsky:Taqqu:1994}
    $$S_{\alpha/2}^*\sim S_{\alpha/2}\left( \left(\frac{\sigma^\alpha C_\alpha}{ C_{\alpha/2}}\int_0^\Delta e^{{\alpha} \vartheta_0 s}ds\right)^{2/\alpha},1,0\right)$$ which implies that $S_{\alpha/2}^*$ is positive (see Proposition 1.2.11 of \cite{Samrodonitsky:Taqqu:1994}).
\end{proof}

\begin{proposition}
	Let the assumptions of \Cref{OU} hold. Then, $W_{OU}$ has a unique minimum in $\vartheta_0$.
\end{proposition}
\begin{proof}
	Proposition 8 of the Supplementary Material of \cite{Fasen:Mayer:2020} implies that under Assumptions \ref{as_D1}, \ref{as_D5} and \ref{as_D7}
% $$\frac 1 {2\vartheta_0}\left(e^{2\vartheta_0\Delta}-1\right)< \frac 1 {4 \pi \vartheta_0}\left(e^{2\vartheta_0\Delta}-1\right)\int_{-\pi}^{\pi}\frac{\left|\Pi^{(\Delta)}(e^{i\omega},\vartheta)\right|^2}{\left|\Pi^{(\Delta)}(e^{i\omega},\vartheta_0)\right|^2}d\omega, \quad \vartheta\neq \vartheta_0,$$ holds. A direct conclusion is then
$$W_{OU}(\vartheta_0)=1<\frac 1 {2\pi}\int_{-\pi}^{\pi}\frac{\left|\Pi^{(\Delta)}(e^{i\omega},\vartheta)\right|^2}{\left|\Pi^{(\Delta)}(e^{i\omega},\vartheta_0)\right|^2}d\omega=W_{OU}(\vartheta) \quad
\text{ for }\vartheta\not=\vartheta_0.$$ Hence, $\vartheta_0$ is indeed the unique minimum.
	\end{proof}

\begin{theorem} \label{consistent:OU}
	Let the assumptions of \Cref{OU} hold. Then, as $n\to \infty$, $$\widehat \vartheta_n^{(\Delta, \alpha)} \overset{\P}{\longrightarrow}\vartheta_0.$$
\end{theorem}
\begin{proof}
	\Cref{OU} and the Skorokhods representation theorem give that there exists a probability space with processes $(W_n^*(\vartheta))_{\vartheta \in \Theta}$ and $(W^*(\vartheta))_{\vartheta \in \Theta}$ having the same distributions as $(W_n^{(\alpha)}(\vartheta))_{\vartheta \in \Theta}$ and $(W^{(\alpha)}(\vartheta))_{\vartheta \in \Theta}$, respectively, with $$\sup_{\vartheta \in \Theta}|W_n^*(\vartheta)-W^*(\vartheta)|\overset{a.s.}{\longrightarrow}0, \quad n\to \infty.$$
With the same arguments as in the proof of Theorem 1 of \cite{Fasen:Mayer:2020}, we can show that the minimizing arguments $\widehat\vartheta_n^*$ and $\widehat\vartheta_0^*$ of $(W_n^*(\vartheta))_{\vartheta \in \Theta}$ and $(W^*(\vartheta))_{\vartheta \in \Theta}$, respectively, satisfy, as $n\to \infty$, $$\vartheta_n^*\overset{a.s.}{\longrightarrow}\vartheta_0,$$ which then implies $\widehat\vartheta_n^{(\Delta,\alpha)}\overset{\mathcal{D}}\longrightarrow \vartheta_0$. Since $\vartheta_0$ is a constant, convergence in distribution implies  convergence in probability.
	\end{proof}

\subsection{Whittle estimation for general symmetric $\alpha$-stable CARMA processes}

\begin{theorem}
Consider the setting of \Cref{consistency3} for a symmetric $\alpha$-stable CARMA(p,q) process with $p\geq 2$. Then, in general, the limit function $W^{(\alpha)}$ of the  Whittle function does not have
a unique minimum in $\vartheta_0$ and hence, the Whittle estimator is not consistent.
\end{theorem}
\begin{proof}
A necessary condition for the  Whittle function $W^{(\alpha)}$ to have a unique minimum in $\vartheta_0$
is that $W^{(\alpha)}(\vartheta)-W^{(\alpha)}(\vartheta_0)$ is a positive random variable for $\vartheta\not=\vartheta_0$
and, hence $\beta_{\vartheta,\vartheta_0}$ as defined in \eqref{Lemma:4.2} is equal to $1$.

However, this is not the case in general as can be seen in \Cref{Bsp1} and \Cref{Bsp2}, which implies that the  Whittle estimator
is in general not consistent.
\end{proof}

\begin{example}\label{Bsp1}~
 We tackle the question, whether it is possible to find a model where $\beta_{\vartheta,\vartheta_0}$ is not equal to 1 for some $\vartheta \neq \vartheta_0$. Therefore, we consider symmetric $3/2$-stable CARMA(2,0) processes with autoregressive and moving average polynomial %$$a_\vartheta(D)Y_t(\vartheta)=c_\vartheta(D)DL_t^{3/2}$$ with
 $$a_\vartheta(z)=z^2-(\vartheta-2)z-2\vartheta\quad \text{and} \quad c_\vartheta(z)=\vartheta-2,$$
 respectively.
These CARMA processes have the state space representation
  \begin{equation*}
{\rm d}X_t(\vartheta) = A(\vartheta) X_t(\vartheta)\,{\rm d}t + e_2\,{\rm d}L_t^{3/2}\quad \text{ and }\quad
Y_t(\vartheta) :=c(\vartheta)^\top X_t(\vartheta),\quad t\geq 0,
\end{equation*} where $$A(\vartheta)=\begin{pmatrix}
0     & 1        \\
 2\vartheta& \vartheta-2         \\
\end{pmatrix} \quad\text{and}\quad c(\vartheta)^\top=(\vartheta-2, 0).$$
  The true parameter is $\vartheta_0=-3$.
The behaviour of $\beta_{\vartheta,\vartheta_0}$ as defined in \Cref{Lemma:4.2}, the behaviour of  the non-normalized positive part $$\beta^+_{\vartheta,\vartheta_0}:={\int_{0}^{\Delta}(G_{\vartheta,\vartheta_0}^+(s))^{\alpha/2}ds}$$ and
the negative part $$ \beta^-_{\vartheta,\vartheta_0}:={\int_{0}^{\Delta}(G_{\vartheta,\vartheta_0}^-(s))^{\alpha/2}ds},$$
respectively,
 are plotted as functions of $\vartheta$ {for $\alpha=1.5$} in \Cref{Fig1}. As one can see, $\beta_{\vartheta,\vartheta_0}^->0$ for all $\vartheta\in (-\infty,-3)\cup (-3,-2)$,
and hence, $\beta_{\vartheta,\vartheta_0}<1$. {Of course, this holds independent of the choice of $\alpha$.} Thus,
$W^{(\alpha)}(\vartheta)-W^{(\alpha)}(\vartheta_0)$ is not a strictly positive random variable for $\vartheta\in (-\infty,-3)\cup (-3,-2)$ and hence,
has not almost surely a unique minimum in $\vartheta_0$. Especially, $\beta_{\vartheta,\vartheta_0}\to 0.8$  for $\vartheta \to -\infty$ {in the case $\alpha=1.5$}.
\begin{figure}[h!]
	\centering
	\includegraphics[scale=0.5]{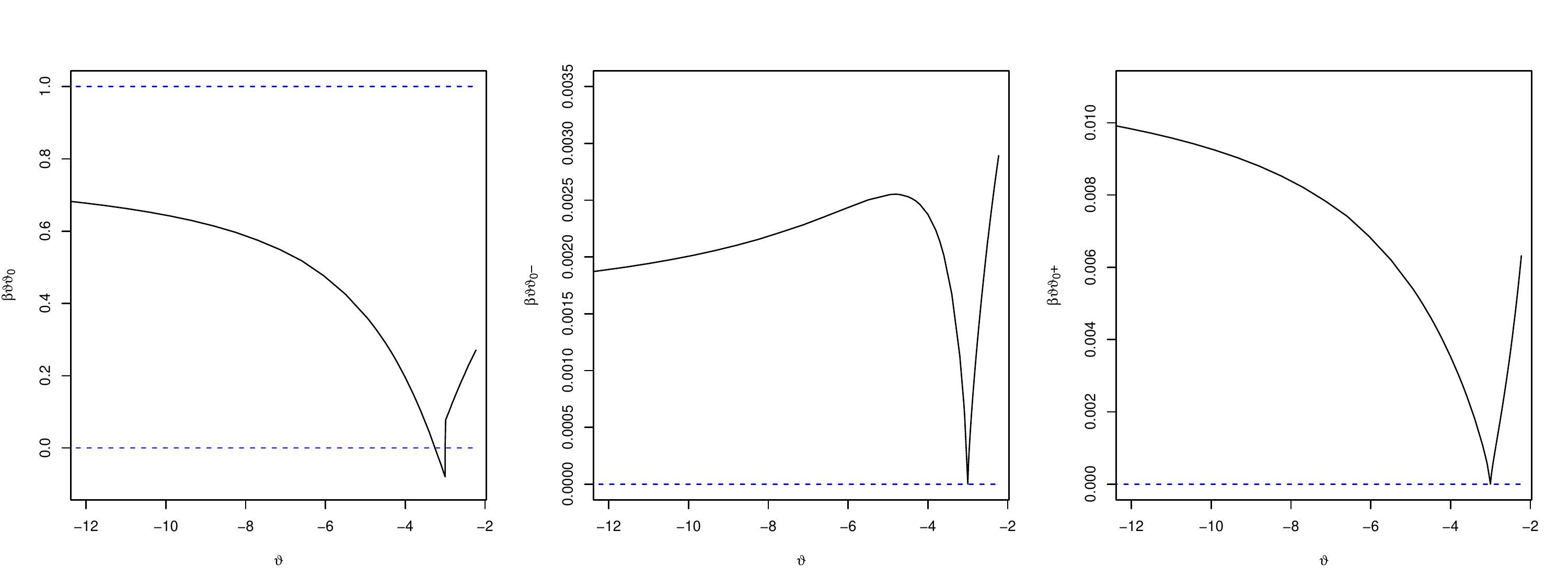}
	\caption{Behaviour of $\beta_{\vartheta,\vartheta_0}, \beta_{\vartheta,\vartheta_0}^-$ and $\beta_{\vartheta,\vartheta_0}^+$ in the CARMA$(2,0)$ model of \Cref{Bsp1}. We set $\beta_{\vartheta_0,\vartheta_0}=0$ to guarantee that $\beta_{\vartheta,\vartheta_0}$ is continuous.} \label{Fig1}
\end{figure}
\end{example}

\begin{example}\label{Bsp2}~
In view of Example 3.3 of \cite{Garcia:Klu:Mueller:2011}, we consider CARMA(2,1) processes with the parametrization \eqref{Equation 1.1}
and $a_1(\vartheta)=\vartheta_1,\ a_2(\vartheta)=\vartheta_2$ and $c_0(\vartheta)=\vartheta_3$, $c_1(\vartheta)=1$, $\vartheta=(\vartheta_1,\vartheta_2,\vartheta_3)\in\Theta$. The kernel function $g_\vartheta$ in \eqref{kernel} has the representation
$$g_\vartheta(t)=\frac{(\vartheta_3-\lambda^+(\vartheta))}{\sqrt{\vartheta_1^2-4\vartheta_2}}e^{-\lambda^+(\vartheta)t}-\frac{(\vartheta_3-\lambda^-(\vartheta))}{{\sqrt{\vartheta_1^2-4\vartheta_2}}}
    e^{-\lambda^-(\vartheta)t},\quad t\geq 0, $$ with $\lambda^+(\vartheta)=\frac{\vartheta_1+\sqrt{\vartheta_1^2-4\vartheta_2}}{2}$ and $\lambda^-(\vartheta)=\frac{\vartheta_1-\sqrt{\vartheta_1^2-4\vartheta_2}}{2} $.  As in \cite{Garcia:Klu:Mueller:2011}, the true parameter is chosen as  $$\vartheta_0=(1.9647,0.0893,0.1761).$$
    Therefore, the kernel function is
$$g_{\vartheta_0}(t)\approx 0.0692 e^{-0.0465t}+0.9307 e^{-1.9181t},\quad t\geq 0,$$ which is non-negative {and we take $\alpha=1.5$}.
In this setting, we calculate $\beta_{\vartheta,\vartheta_0}$ as a function of the components $\vartheta_1, \vartheta_2$ and $\vartheta_3$, respectively, where we fix the other two variables. Then the functions
$\beta_{\vartheta,\vartheta_0},\beta^-_{\vartheta,\vartheta_0}$ and $\beta^+_{\vartheta,\vartheta_0}$ are plotted in \Cref{Figure2}. In all three cases, the plots show that $\beta^-_{\vartheta,\vartheta_0}>0$  for some $\vartheta \neq\vartheta_0$ implying $\beta_{\vartheta,\vartheta_0}<1$. Therefore,  if we only
allow a single parameter to vary, the  Whittle estimator converges to a function which has not an unique minimum in the true parameter. Hence,  the  Whittle estimator is not consistent.
Again this statement is independent of the choice of $\alpha$.
\begin{figure}[h!]
	\centering
	\includegraphics[scale=0.5]{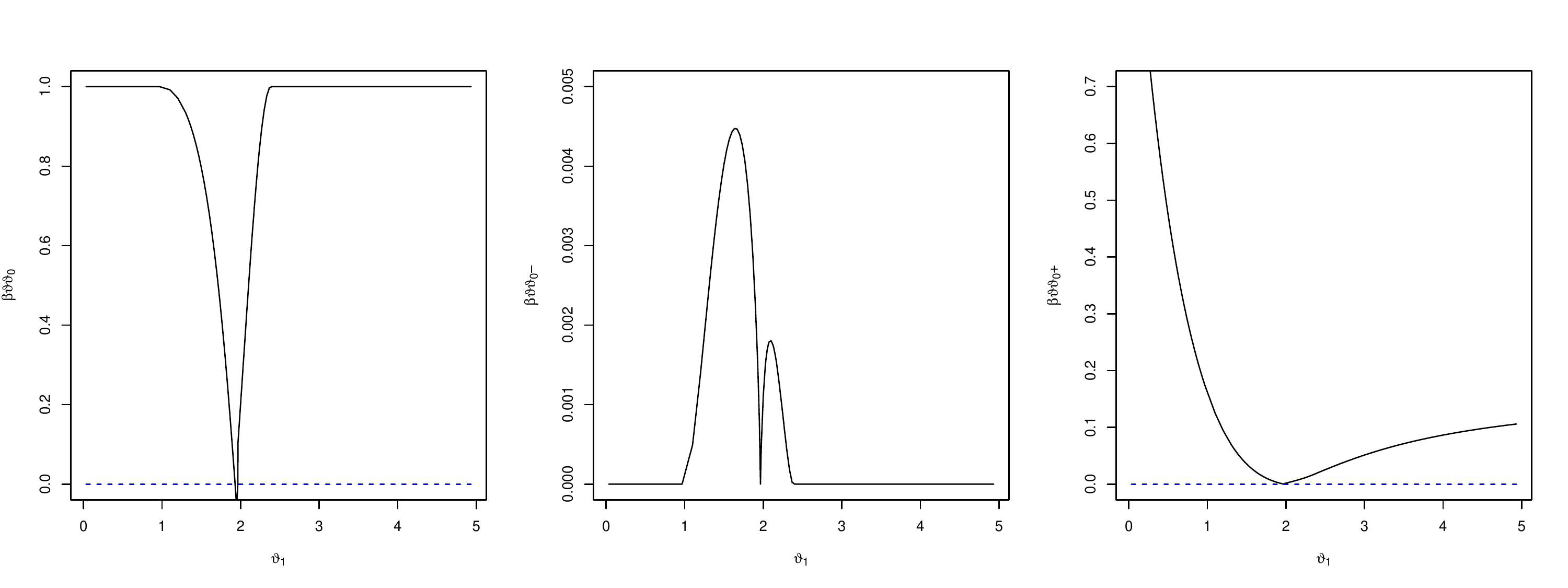}
	\includegraphics[scale=0.5]{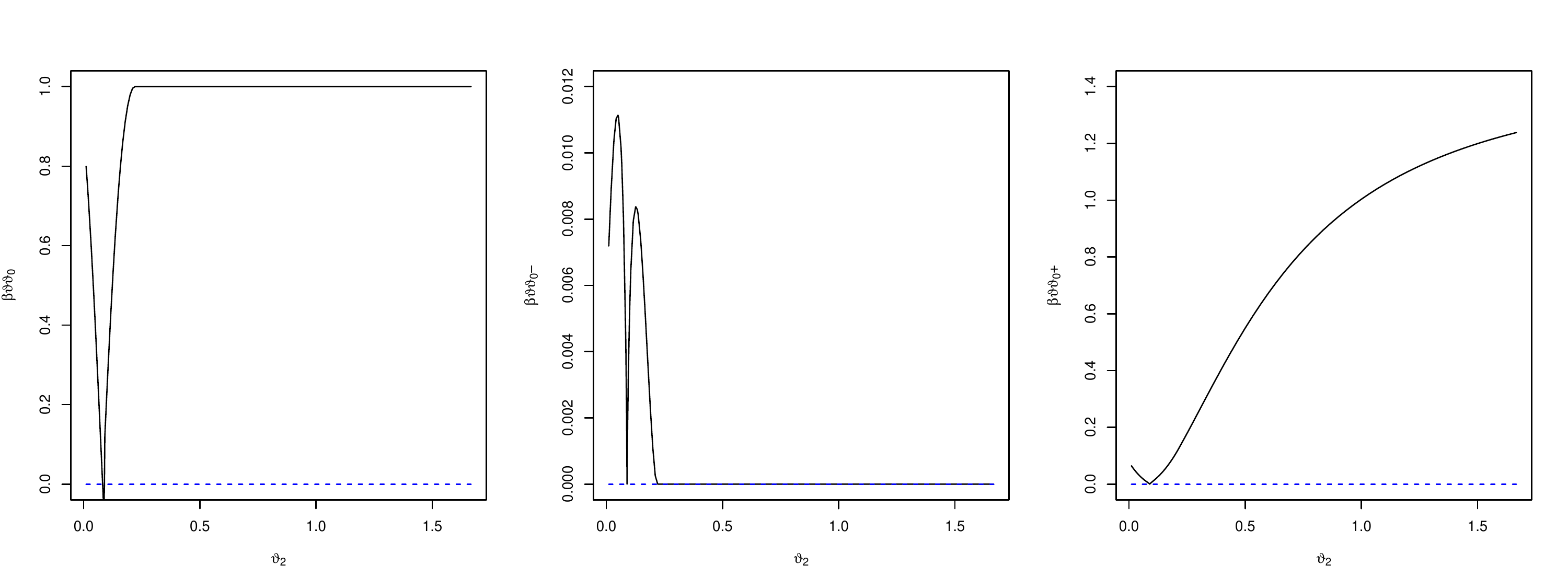}
	\includegraphics[scale=0.5]{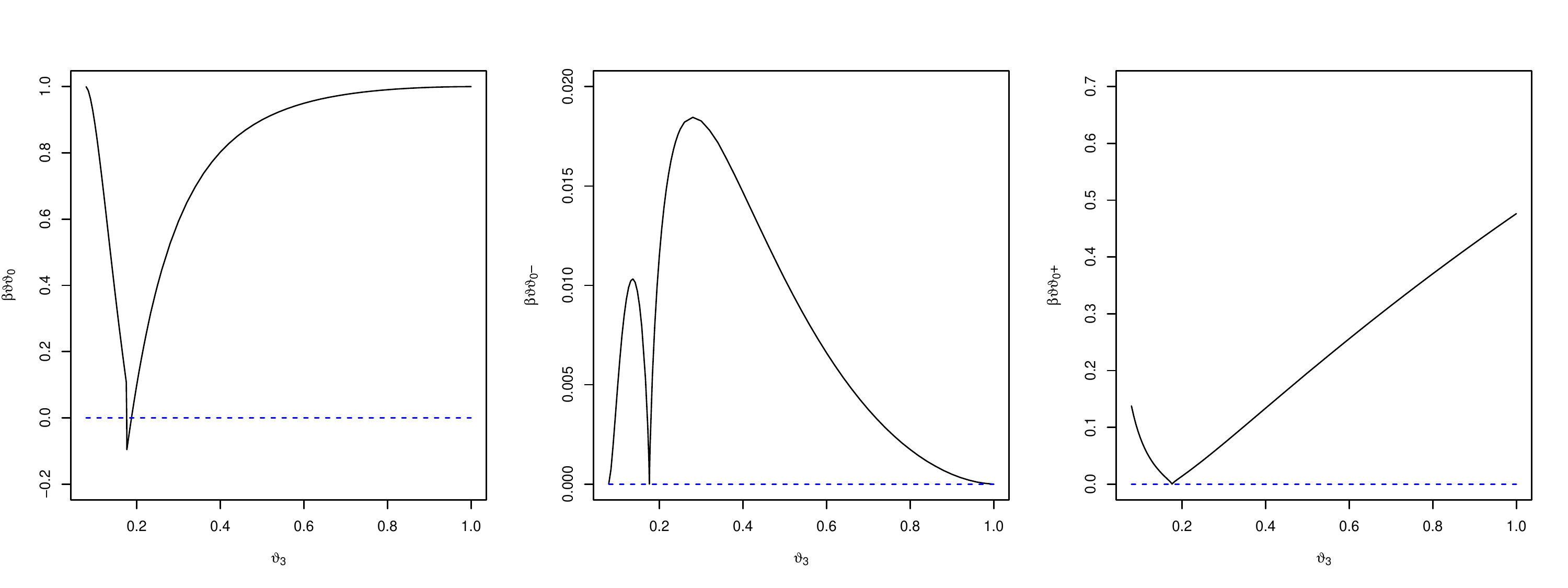}
	\caption{behaviour of $\beta_{\vartheta,\vartheta_0}, \beta_{\vartheta,\vartheta_0}^-$ and $\beta_{\vartheta,\vartheta_0}^+$ in the CARMA$(2,1)$ model of \Cref{Bsp2} where $\vartheta$ originates from $\vartheta_0$ when we fix two components and vary the third one. We set $\beta_{\vartheta_0,\vartheta_0}=0$ to guarantee that $\beta_{\vartheta,\vartheta_0}$ is continuous.} \label{Figure2}
\end{figure}
\end{example}
\section{Simulation study}\label{Sec5}
In this section, we investigate the performance of the  Whittle estimator for finite samples through a simulation study and compare it
with the behaviour of the estimator introduced in \cite{Garcia:Klu:Mueller:2011}.

The estimator of \cite{Garcia:Klu:Mueller:2011} is based on an indirect approach. Denoting the zeros of the AR$(p)$  polynomial $a(z)$ as
$\lambda_1,\ldots,\lambda_p$ which are assumed to be distinct and defining $a_D^{(\Delta)}(z)=\prod_{j=1}^{p}(1- e^{\lambda_j\Delta}z)$, it is well known \cite[Proposition 3.1]{Garcia:Klu:Mueller:2011} that  the sampled process $Y^{(\Delta)}$ satisfies the equation
\begin{eqnarray}\label{dependent_sampled}
    a_D^{(\Delta)}(\mathsf{ B})Y_k^{(\Delta)}=U_k^{(\Delta)},\quad k \in \N,
    \end{eqnarray}
    where $(U_k^{(\Delta)})_{k\in \N}$ is a $(p-1)$-dependent sequence and $\mathsf{ B}$ is the backshift operator with $\mathsf{ B}Y_k^{(\Delta)}=Y_{k-1}^{(\Delta)}$.
    For  CARMA processes with finite second moments,  $(U_k^{(\Delta)})_{k\in\N}$ is a MA$(p-1)$ process such that $Y^{(\Delta)}$ is an ARMA$(p,p-1)$
    process with an uncorrelated but not independent white noise.
    \cite{Garcia:Klu:Mueller:2011} proposed to fit an ARMA$(p,p-1)$ model to the observations $Y_1^{(\Delta)},\ldots,Y_n^{(\Delta)}$ by standard maximum likelihood
    estimation for Gaussian ARMA models. The estimated autoregressive part of that ARMA model in discrete time is denoted by $\widehat a_D^{(\Delta)}(z)$ and the estimated
    moving average part we denote by $\widehat c_D^{(\Delta)}(z)$. The logarithmic zeros of $\widehat a_D^{(\Delta)}(z)$ divided by $-\Delta$
    are then estimators $\widehat\lambda_1,\ldots,\widehat\lambda_p$
    for the  zeros $\lambda_1,\ldots,\lambda_p$ of $a(z)$. Hence, we obtain an estimator $\widehat a(z)$ for the autoregressive polynomial $a(z)$.
     In a final step,  the MA polynomial $c(z)$ of the CARMA process is determined. Therefore the parameter $\vartheta=(\vartheta_1,\vartheta_2)$ is divided
     in two parts where $\vartheta_1$ models the AR coefficients and $\vartheta_2$ the MA coefficients of the CARMA process. Now
      the autocorrelation function $\rho_{\widehat a,\vartheta_2}^{\text{(MA)}}$  of
    $ \widehat a_D^{(\Delta)}(\mathsf{ B})Y^{(\Delta)}(\widehat \vartheta_1,\vartheta_2)$  and the autocorrelation function $\rho_{\widehat c_D^{(\Delta)}}^{\text{(MA)}}$
    of  a discrete time moving average process with moving average polynomial  $\widehat c_D^{(\Delta)}(z)$ is calculated and $\widehat\vartheta_2$
    is derived  numerically as solution of $\rho_{\widehat a_D^{(\Delta)},\widehat\vartheta_2}^{\text{(MA)}}(k)=\rho_{\widehat c_D^{(\Delta)}}^{\text{(MA)}}(k)$
    for $k=1,\ldots,q$.

In the following, we use an Euler-Maruyama scheme for differential equations with initial value $Y_0=0$ and step size $0.01$ to simulate $\alpha$-stable
CARMA processes. We set $\Delta=1$ as the distance between the discrete observations and $\alpha=1.5$ for the stable index of the driving symmetric $\alpha$-stable Lévy process.
We investigate the behaviour of the Whittle estimator
and the estimator of \cite{Garcia:Klu:Mueller:2011}  for $n=500, 2000, 5000$ based on 500 replications.

 As first example, we simulate an Ornstein-Uhlenbeck process with $\vartheta_0=-1$. The resulting sample mean, bias and sample standard deviation are given in \Cref{table_1}.
\begin{table}[h]	
	\begin{center}
		\begin{tabular}{|c||c|c|c||c|c|c||c|c|c|}\hline
			\multicolumn{10}{|c|}{Whittle, $\alpha=1.5$} \\\hline
			&  \multicolumn{3}{|c||}{$n=500$} & \multicolumn{3}{|c||}{$n=2000$}& \multicolumn{3}{|c|}{$n=5000$}\\ \hline
			\hspace*{0.1cm} \hspace*{0.1cm} &mean & bias& std.& mean & bias &std. & mean & bias &std. \\ \hline
			$\vartheta_0=-1$ &-1.0132&0.0132&0.1118&-1.0082&0.0082&0.0528&-1.0071&0.0071&0.0367  \\ \hline\hline
		\multicolumn{10}{|c|}{Estimator of Garc\'{\i}a et al., $\alpha=1.5$} \\\hline
		&  \multicolumn{3}{|c||}{$n=500$} & \multicolumn{3}{|c||}{$n=2000$}& \multicolumn{3}{|c|}{$n=5000$}\\ \hline
		\hspace*{0.1cm} \hspace*{0.1cm} &mean & bias& std.& mean & bias &std.& mean & bias &std.  \\ \hline
		$\vartheta_0=-1$ &-1.0162&0.0162&0.1018&-0.9948&0.0052&0.0522&-0.9942&0.0058&0.0333  \\ \hline
				\end{tabular}
	\end{center}
	\label{table5}
	\begin{center}
		\caption{Estimation results for a symmetric $1.5$-stable Ornstein-Uhlenbeck process with parameter $\vartheta_0=-1$. }  \label{table_1} 	
	\end{center}
\end{table}
It seems that both the Whittle estimator and the estimator of \cite{Garcia:Klu:Mueller:2011} converge to the true value. For the Whittle estimator this is consistent with \Cref{consistent:OU}.
To compare the behaviour in the heavy-tailed setting with the behaviour in the light-tailed setting, we present a second simulation
study where we  use for the driving Lévy process of the Ornstein-Uhlenbeck model  a Brownian motion. The results are given in \Cref{table_2}. As we can see, the behaviour of the  Whittle estimator and  the estimator of \cite{Garcia:Klu:Mueller:2011} are similar for the light-tailed and for the heavy-tailed Ornstein-Uhlenbeck process.
%However, we would obviously recommend using the  Whittle estimator or the estimator of Garc\'{\i}a et al. for parameter estimation of a symmetric $\alpha$-stable Ornstein-Uhlenbeck process.

\begin{table}[h]	
	\begin{center}
		\begin{tabular}{|c||c|c|c||c|c|c||c|c|c|}\hline
			\multicolumn{10}{|c|}{Whittle, $\alpha=2$} \\\hline
			&  \multicolumn{3}{|c||}{$n=500$} & \multicolumn{3}{|c||}{$n=2000$}& \multicolumn{3}{|c|}{$n=5000$}\\ \hline
			\hspace*{0.1cm} \hspace*{0.1cm} &mean & bias& std.& mean & bias &std. & mean & bias &std. \\ \hline
			$\vartheta_0=-1$ &-1.0143&0.0143&0.1183&-1.0082&0.0082&0.0528&-1.0002&0.0002&0.0349  \\ \hline\hline
			\multicolumn{10}{|c|}{Estimator of Garc\'{\i}a et al., $\alpha=2$} \\\hline
			&  \multicolumn{3}{|c||}{$n=500$} & \multicolumn{3}{|c||}{$n=2000$}& \multicolumn{3}{|c|}{$n=5000$}\\ \hline
			\hspace*{0.1cm} \hspace*{0.1cm} &mean & bias& std.& mean & bias &std.& mean & bias &std.  \\ \hline
			$\vartheta_0=-1$ &-1.0007&0.0007&0.1133&-1.0011&0.0011&0.0568&-1.0012&0.0012&0.0351  \\ \hline
				\end{tabular}
	\end{center}
	\label{table5}
	\begin{center}
		\caption{Estimation results for a Brownian motion driven Ornstein-Uhlenbeck process with parameter $\vartheta_0=-1$. }  \label{table_2} 	
	\end{center}
\end{table}

Next, we simulate the CARMA(2,0) process of \Cref{Bsp1}. Accordingly, the true value is  $\vartheta_0=-3$.
 The results are given in \Cref{table_3}.
As already argued in
 \Cref{Bsp1} the Whittle estimator is not a consistent estimator in this situation. This is confirmed
 by the simulation study. For $n=5000$ the bias and standard deviation are even higher than for $n=2000$. The estimator of \cite{Garcia:Klu:Mueller:2011}
 behaves even worse. On the one hand, the bias and standard deviation of \cite{Garcia:Klu:Mueller:2011} are quite high and not decreasing with increasing sample size. On the other hand, the estimation procedure of \cite{Garcia:Klu:Mueller:2011} stops for every sample size for more than 1/5th of the replications. This can be traced back to an inadequate estimate of the zero of the AR polynomial, namely the real part of the estimated zero of the AR polynomial is less than 0 which means that the logarithm of this zero is not defined.

Finally, we investigate the CARMA(2,1) process of \Cref{Bsp2}, see \Cref{table_4}.
Our simulation results show the same findings as \cite{Garcia:Klu:Mueller:2011};  both estimators perform very well in this parameter setting.
However, most of the time there is one parameter which has a slightly higher bias or standard deviation
 such that it is not apparent if the estimator is converging. Indeed, for the Whittle estimator we already showed
 in \Cref{Bsp2} that this is not the case and we guess that the same holds true for the estimator of  \cite{Garcia:Klu:Mueller:2011}, although
 at the first view this seems to contradict the simulation study.
  But from the behaviour of $\beta_{\vartheta,\vartheta_0}$ in \Cref{Figure2} we know that only in a small neighbourhood of $\vartheta_0$, the random variables
 $W^{(\alpha)}(\vartheta)-W^{(\alpha)}(\vartheta_0)$ are not positive and outside this neighbourhood they are positive with probability one because
 $\beta_{\vartheta,\vartheta_0}=1$.  Although $W^{(\alpha)}(\vartheta)$
 has not a unique minimum in $\vartheta_0$, $\vartheta_0$ is close to the minimum of  $W^{(\alpha)}(\vartheta)$. Thus, the Whittle
 estimator is close to the true value $\vartheta_0$ as well.

\begin{table}	[h]
	\begin{center}
		\begin{tabular}{|c||c|c|c||c|c|c||c|c|c|}\hline
			\multicolumn{10}{|c|}{Whittle, $\alpha=1.5$} \\\hline
			&  \multicolumn{3}{|c||}{$n=500$} & \multicolumn{3}{|c||}{$n=2000$}& \multicolumn{3}{|c|}{$n=5000$}\\ \hline
			\hspace*{0.1cm} \hspace*{0.1cm} &mean & bias& std.& mean & bias &std. & mean & bias &std. \\ \hline
			$\vartheta_0=-3$ &-3.4762&0.4762&1.2741&-3.2902&0.2902& 0.9367&-3.3002&0.3002&0.9568  \\ \hline\hline
			\multicolumn{10}{|c|}{Estimator of Garc\'{\i}a et al., $\alpha=1.5$} \\\hline
			&  \multicolumn{3}{|c||}{$n=500$ }& \multicolumn{3}{|c||}{$n=2000$ }& \multicolumn{3}{|c|}{$n=5000$ }\\ \hline
			\hspace*{0.1cm} \hspace*{0.1cm} &mean & bias& std.& mean & bias &std.& mean & bias &std.  \\ \hline
			$\vartheta_0=-3$ &-3.2473& 0.2473& 1.2220&-3.8184&0.8164&1.1089&-4.0770&1.0770&0.9238 \\ \hline
					\end{tabular}
	\end{center}
	\begin{center}
		\caption{Estimation results for the symmetric $1.5$-stable CARMA(2,0) process of \Cref{Bsp1}.  } \label{table_3}
	\end{center}
\end{table}

\begin{table}[h]	
	\begin{center}
		\begin{tabular}{|c||c|c|c||c|c|c||c|c|c|}\hline
			\multicolumn{10}{|c|}{Whittle, $\alpha=1.5$} \\\hline
			&  \multicolumn{3}{|c||}{$n=500$} & \multicolumn{3}{|c||}{$n=2000$}& \multicolumn{3}{|c|}{$n=5000$}\\ \hline
			\hspace*{0.1cm} \hspace*{0.1cm} &mean & bias& std.& mean & bias &std. & mean & bias &std. \\ \hline
			$\vartheta_1=1.9647$ &1.9520& 0.0127&0.0516&1.9592&0.0055&0.0321&2.0069&0.0422&1.1890 \\
		$\vartheta_2=0.0893$&0.1031&0.0138&0.0377&0.0940&0.0047&0.0224&0.0987 &0.0094&0.0288 \\
		$\vartheta_3=0.1761$&-0.0144&0.1905&0.1836&-0.0389 &0.215&0.1681& 0.1735&0.0026&0.0224 \\ \hline\hline
			\multicolumn{10}{|c|}{Estimator of Garc\'{\i}a et al., $\alpha=1.5$} \\\hline
			&  \multicolumn{3}{|c||}{$n=500$ } & \multicolumn{3}{|c||}{$n=2000$ }& \multicolumn{3}{|c|}{$n=5000$ }\\ \hline
			\hspace*{0.1cm} \hspace*{0.1cm} &mean & bias& std.& mean & bias &std.& mean & bias &std.  \\ \hline
			$\vartheta_1=1.9647$&2.0947&0.1300&0.4480 &2.0138&0.0491&0.2405&2.0036&0.0389&0.1543 \\
				$\vartheta_2=0.0893$ &0.1462&0.0569&0.2160&0.0939& 0.0046&0.0323&0.0930&0.0037&0.0300 \\
					$\vartheta_3=0.1761$&0.2196& 0.0435&0.1333&0.1877&0.0116&0.0487&0.1920& 0.0159&0.0484\\\hline
				\end{tabular}
	\end{center}
	\label{table5}
	\begin{center}
		\caption{Estimation results for the symmetric $1.5$-stable CARMA(2,1) process of \Cref{Bsp2}. }  \label{table_4} 	
	\end{center}
\end{table}

\section{Conclusion} \label{Conclusion}

The simulation study confirms the theory that  for symmetric $\alpha$-stable CARMA$(p,q)$ processes with $p\geq 2$
the  Whittle estimator is in general not consistent. Similarly, it suggests that the estimator of
\cite{Garcia:Klu:Mueller:2011}  is  not a consistent estimator as well, although for the special parameter setting
of \Cref{Bsp2} both, the  Whittle estimator and the estimator of \cite{Garcia:Klu:Mueller:2011},  give quite reasonable results.
In case of the  Whittle estimator this effect is not surprising. \Cref{Figure2} suggests that $\P(W^{(\alpha)}(\vartheta)>W^{(\alpha)}(\vartheta_0))$
is quite high for $\vartheta\not=\vartheta_0$ such that the probability of a local minimum in $\vartheta_0$ is still high, even though  $W^{(\alpha)}$
has not a unique minimum in $\vartheta_0$.

Essentially, the  Whittle estimator is not a consistent for symmetric $\alpha$-stable CARMA$(p,q)$ processes since $(\varepsilon_k^{(\Delta)}(\vartheta))_{k\in\N }$, the noise in the MA representation \eqref{linearInnovations}, is a dependent sequence except for Ornstein-Uhlenbeck processes. Then, it is an iid sequence
and therefore, it is not astonishing that the  Whittle estimator is consistent. If the driving Lévy process has finite second moments $(\varepsilon_k^{(\Delta)}(\vartheta))_{k\in\N }$
is at least an uncorrelated sequence which is sufficient for the consistency of the Whittle estimator for CARMA processes
(see \Cref{Theorem 3.1}).

For similar reasons, we also assume that the estimator of \cite{Garcia:Klu:Mueller:2011}
is not a consistent estimator in general.
In fact, if the driving Lévy process has finite second moments, the process $(U_k^{(\Delta)})_{k\in\N}$ given in
\eqref{dependent_sampled} is a MA$(p-1)$ process with a weak white noise.
In contrast, for $\alpha$-stable CARMA$(p,q)$ models with $p\geq 2$,  the sequence  $(U_k^{(\Delta)})_{k\in\N}$
has not a representation as a \linebreak MA$(p-1)$ process with an iid noise  and due to the lack of the second
moments the sequence can not be uncorrelated as well. This dependence influences the estimator and prevents consistency.
However, if $Y$ is an Ornstein-Uhlenbeck process,
$(U_k^{(\Delta)})_{k\in\N}$ is an iid sequence which implies that the estimator of \cite{Garcia:Klu:Mueller:2011}
is consistent as well. The simulation results in \Cref{table_1} matches this insight.
% Furthermore, the estimation procedure  of  \cite{Garcia:Klu:Mueller:2011} uses the empirical correlation of ???. But
%in contrast to discrete-time ARMA processes the correlation function is not consistent estimator for the quasi correlation function;
%this is a direct conclusion of \Cref{CorDrapatz} (for more details see \cite{Drapatz:2017}).

In sum, a Lévy driven CARMA process with finite second moments sampled discretely has a weak ARMA representation,
which can be used to estimate the CARMA parameters as, e.g., for quasi-maximum likelihood estimation (see \cite{QMLE}) or for Whittle estimation
(see \cite{Fasen:Mayer:2020}). Beside the identifiability issue, a difficulty in this attempt is that the noise of the weak ARMA representation depends
on the model parameters as we already saw for the Ornstein-Uhlenbeck process in \Cref{sec:4.1}.
However, for heavy-tailed CARMA processes with infinite second moments
the estimators for heavy-tailed ARMA models as, e.g., the Whittle estimator do not work anymore because the noise in the ARMA representation
of the discretely sampled CARMA process is neither independent nor uncorrelated. The same phenomena was also investigated for the parameter estimation of GARCH(1,1) processes with  infinite
4th moment in \cite{Mikosch:Straumann:2002} where the Whittle estimator is inconsistent although it is consistent in the finite 4th moment case
(see \cite{Giraitis:Robinson:2001} and \cite{Mikosch:Straumann:2002}). There as well the noise of the ARMA representation of the squared GARCH(1,1) process with finite 4th moments
is   a weak white noise which is in general not independent.

In conclusion, the analogy between parameter estimation for heavy-tailed CARMA and heavy-tailed ARMA processes with infinite variance is dangerous.
For the estimation of heavy-tailed CARMA models different estimation approaches as for ARMA processes have to be developed,
although they work for CARMA processes with finite second moments. This is topic of some future research.

\begin{appendix}
\section{Asymptotic behaviour of the sample autocovariance function of symmetric $\alpha$-stable CARMA processes}

\begin{proposition}\label{GenDrapatz}
	Let
$g_1,g_2:\R\to\R$ be bounded functions with $g_1,g_2\in L^{\delta}(\R)$ for some $\delta<\min\{\alpha,1\}$ and
$0<\int_{-\infty}^{\infty}|g_1(s)g_2(s)|\,ds<\infty$.
Suppose $L^{(\alpha)}$ is a symmetric $\alpha$-stable L\'evy process with $\alpha \in (0,2)$ and $L_1^\alpha \sim S_\alpha(\sigma,0,0)$. Define the
continuous-time MA processes
$$Y_t^{[1]}=\int_{-\infty}^{\infty} g_1(t-s)dL_s^{(\alpha)} \quad \text{ and } \quad Y_t^{[2]}=\int_{-\infty}^{\infty} g_2(t-s)dL_s^{(\alpha)}, \quad t\geq 0.$$
Furthermore,   $G_{g_1,g_2}:[0,\Delta]\to \R$ is given as $s\to \sum_{j=-\infty}^{\infty}g_1(\Delta j -s)g_2(\Delta j-s)$ and suppose $G_{g_1,g_2}\in L^{(\alpha/2)}[0,\Delta]$.
Then, as $n\to\infty$,
$$\frac 1 {n^{2/\alpha}}\sum_{k=1}^{n}Y_{k\Delta }^{[1]} Y_{k\Delta }^{[2]}\overset{\mathcal D}{\longrightarrow}\int_0^\Delta G_{g_1,g_2}(s) dL_s^{(\alpha/2)},$$
where $L^{(\alpha/2)}$ is the $\alpha/2$-stable L\'evy process of \Cref{consistency3}.
	\end{proposition}
\begin{proof}
	The proof is mostly the same as the proof of the asymptotic behaviour of the sample autocovariance function of a
continuous-time moving average process in Theorem 3 of \cite{Drapatz:2017}
{and is therefore omitted.}
\end{proof}

	\begin{remark}\label{RemDrapatz}
		Note that  $$\int_0^\Delta G_{g_1,g_2}(s) dL_s^{(\alpha/2)} \sim S_{\alpha/2}(\sigma_{g_1,g_2},\beta_{g_1,g_2},0),$$
		 is an ${\alpha/2}$-stable distribution with parameters
    \begin{eqnarray*}
		\beta_{g_1,g_2}&=&\frac{\int_{0}^{\Delta}(G_{g_1,g_2}^+(s))^{\alpha/2}-(G_{g_1,g_2}^-(s))^{\alpha/2}ds}{\int_0^{\Delta}|G_{g_1,g_2}(s)|^{\alpha/2}ds},\\
		\sigma_{g_1,g_2}^{\alpha/2}&=&\frac{\sigma^{\alpha}C_\alpha}{C_{\alpha/2}} \int_0^\Delta |G_{g_1,g_2}(s)|^{\alpha/2}ds,
    \end{eqnarray*}
 see Property 1.2.3 and 3.2.2 of \cite{Samrodonitsky:Taqqu:1994}.
	\end{remark}

	\begin{theorem}\label{CorDrapatz}
		Let $Y$ be a symmetric $\alpha$-stable CARMA process with kernel function \break $g(t)=c^\top\e^{A t}e_p\1_{\left[0,\infty\right)}(t)$ as given in \eqref{kernel} and $\overline \gamma_{n}(h)$, $h=-n+1,\ldots,n-1$ be the sample autocovariance function as
defined in \eqref{ACF}. Then, for fixed $m\in \N$ and as $n\to\infty$,
\beao
\lefteqn{\frac{1}{n^{2/\alpha-1}}\left(\overline \gamma_{n}(0),\ldots,\overline \gamma_n(m)\right)}\\
    &&\overset{\mathcal D}{\longrightarrow}\left(\int_0^\Delta \sum_{j=-\infty}^{\infty}g(\Delta j-s)^2 dL^{(\alpha/2)}_s,\ldots, \int_0^\Delta \sum_{j=-\infty}^{\infty}g(\Delta j-s)g(\Delta(j+m)-s) dL^{(\alpha/2)}_s\right),
\eeao where $L^{(\alpha/2)}$ is the ${\alpha}/{2}$-stable Lévy process of \Cref{consistency3}.
	\end{theorem}
	\begin{proof}
		Let $c_0,\ldots, c_m \in \R$. Then,
        \begin{eqnarray} \label{eq1}
        \frac{n}{n^{2/\alpha}}\left(c_0\overline \gamma_{n}(0)+\ldots+c_m\overline \gamma_n(m)\right)&=&\frac{1}{n^{2/\alpha}}\left(c_0\sum_{j=1}^{n}Y_j^{(\Delta)2}+\ldots+ c_m\sum_{j=1}^{n-m}Y_j^{(\Delta)}Y_{j+m}^{(\Delta)}\right) \nonumber\\ &=&\frac{1}{n^{2/\alpha}}\left(\sum_{j=1}^{n}Y_j^{(\Delta)}\left(\sum_{k=0}^{m}c_kY_{k+j}^{(\Delta)}\right)-\sum_{j=n-m+1}^{n}\sum_{k=n-j+1}^{m}c_kY_j^{(\Delta)}Y_{k+j}^{(\Delta)}\right)\nonumber\\
        &=:&J_n^{[1]}+J_n^{[2]}.
		\end{eqnarray}		
		We obtain for $\delta <\alpha/2$ that
		\begin{eqnarray*}
        \E\left|J_n^{[2]}\right|^{\delta}=\E\left|n^{-2/\alpha}\sum_{j=n-m+1}^{n}\sum_{k=n-j+1}^{m}c_kY_j^{(\Delta)}Y_{k+j}^{(\Delta)}\right|^{\delta}
        \leq n^{-2\delta/\alpha} m^{2\delta} \max_{k=0,\ldots, m}|c_k|^\delta \E\left|Y_1^{(\Delta)}Y_{1+k}^{(\Delta)}\right|^{\delta}
        \overset{n\to \infty}{\longrightarrow}0.
        \end{eqnarray*}
		Therefore, the second term $J_n^{[2]}$ in \eqref{eq1} is negligible. For the first term $J_n^{[1]}$ in \eqref{eq1} we define $$
		Y_{t}^{[1]}:=\int_{-\infty}^{\infty}\sum_{k=0}^{m}c_kg(t+k\Delta-s)d L_s^{(\alpha)} \quad \text{ and } \quad Y_{t}^{([2]}:=\int_{-\infty}^{\infty}g(t-s)d L_s^{(\alpha)}, \quad t\geq 0.$$
		Thereby, we have
        \begin{eqnarray*}
        \sum_{k=0}^{m}c_kY_{k+j}^{(\Delta)}&=&\int_{-\infty}^{\infty}\sum_{k=0}^{m}c_kg((k+j)\Delta-s)d L_s^{(\alpha)}=Y_{j\Delta}^{[1]},\\
		Y_j^{(\Delta)}&=&\int_{-\infty}^{\infty}g(j\Delta-s)d L_s^{(\alpha)}=Y^{[2]}_{j\Delta}.
		\end{eqnarray*}
		An application of
		\Cref{GenDrapatz} leads for $n\to \infty$ to
		\begin{eqnarray*}
		\frac{n}{n^{2/\alpha}}\sum_{k=0}^{m}c_k\overline \gamma_n(k)&=&\frac 1 {n^{2/\alpha}}\sum_{k=1}^{n}Y_{k\Delta }^{[1]} Y_{k\Delta }^{(2)}+J_n^{[2]}\\
        &\overset{\mathcal D}{\longrightarrow}&\int_0^{\Delta}\sum_{j=-\infty}^{\infty}\left(\sum_{k=0}^{m}c_kg(\Delta (k+j)-s)g(\Delta j -s)\right)dL_s^{(\alpha/2)}\\
		&=&\sum_{k=0}^{m}\int_0^{\Delta}\sum_{j=-\infty}^{\infty}c_kg(\Delta (k+j)-s)g(\Delta j -s)dL_s^{(\alpha/2)}.
		\end{eqnarray*}
		Cram\'er-Wold completes the proof.
	\end{proof}

\end{appendix}
%%%%%%%%%%%%%%%%%%%%%%%%%%%%%%%%%%%%%%%%%%%%%%%%%%%%%%%%%%%%%%%%%

{\bibliography{ForschungMayer2}}

\begin{thebibliography}{25}
\providecommand{\natexlab}[1]{#1}
\providecommand{\url}[1]{\texttt{#1}}
\expandafter\ifx\csname urlstyle\endcsname\relax
  \providecommand{\doi}[1]{doi: #1}\else
  \providecommand{\doi}{doi: \begingroup \urlstyle{rm}\Url}\fi

\bibitem[Billingsley(1999)]{Billingsley:1999}
P.~Billingsley.
\newblock \emph{Convergence of Probability Measures}.
\newblock Wiley, New York, 1999.

\bibitem[Brockwell et~al.(2011)Brockwell, Davis, and Yang]{brockwelldavisyang}
P.~Brockwell, R.~Davis, and Y.~Yang.
\newblock Estimation for non-negative {L}\'evy-driven {CARMA} processes.
\newblock \emph{J. Bus. Econom. Statist.}, 29\penalty0 (2):\penalty0 250--259,
  2011.

\bibitem[Brockwell and Davis(1991)]{Brockwell:Davis:1991}
P.~J. Brockwell and R.~A. Davis.
\newblock \emph{{Time Series: Theory and Methods}}.
\newblock Springer Series in Statistics, New York, 1991.

\bibitem[Brockwell and Lindner(2009)]{Brockwell:Lindner:2009}
P.~J. Brockwell and A.~Lindner.
\newblock {Existence and uniqueness of stationary L{\'e}vy-driven {CARMA}
  processes}.
\newblock \emph{Stochastic Process. Appl.}, 119\penalty0 (8):\penalty0
  2660--2681, 2009.

\bibitem[Brockwell and Lindner(2019)]{Brockwell:Lindner:2019}
P.~J. Brockwell and A.~Lindner.
\newblock Sampling, embedding and inference for {CARMA} processes.
\newblock \emph{J. Time Ser. Anal.}, 40\penalty0 (2):\penalty0 163--181, 2019.

\bibitem[Chui and Chen(2009)]{Chui:Chen:2009}
C.~K. Chui and G.~Chen.
\newblock \emph{Kalman filtering: with real-time applications}.
\newblock Springer, Berlin, 4th edition, 2009.

\bibitem[Drapatz(2017)]{Drapatz:2017}
M.~Drapatz.
\newblock Limit theorems for the sample mean and sample autocovariances of
  continuous time moving averages driven by heavy-tailed {L}\'{e}vy noise.
\newblock \emph{ALEA Lat. Am. J. Probab. Math. Stat.}, 14\penalty0
  (1):\penalty0 403--426, 2017.

\bibitem[Dunsmuir and Hannan(1976)]{DunsmuirHannan76}
W.~Dunsmuir and E.~J. Hannan.
\newblock {Vector linear time series models}.
\newblock \emph{Adv. in Appl. Probab.}, 8\penalty0 (2):\penalty0 339--364,
  1976.

\bibitem[Fasen(2013{\natexlab{a}})]{Fasen:2013}
V.~Fasen.
\newblock Statistical estimation of multivariate {Ornstein-Uhlenbeck} processes
  and applications to co-integration.
\newblock \emph{J. Econometrics}, 172\penalty0 (2):\penalty0 325--337,
  2013{\natexlab{a}}.

\bibitem[Fasen(2013{\natexlab{b}})]{fasen2013a}
V.~Fasen.
\newblock {Statistical inference of spectral estimation for continuous-time MA
  processes with finite second moments}.
\newblock \emph{Math. Methods Statist.}, 22\penalty0 (4):\penalty0 283--309,
  2013{\natexlab{b}}.

\bibitem[Fasen-Hartmann and Mayer(2020)]{Fasen:Mayer:2020}
V.~Fasen-Hartmann and C.~Mayer.
\newblock Whittle estimation of state space models with finite second moments.
\newblock \emph{Submitted}, 2020.
\newblock arXiv:2002.09426.

\bibitem[Garc\'{\i}a et~al.(2011)Garc\'{\i}a, Kl\"{u}ppelberg, and
  M\"{u}ller]{Garcia:Klu:Mueller:2011}
I.~Garc\'{\i}a, C.~Kl\"{u}ppelberg, and G.~M\"{u}ller.
\newblock Estimation of stable {CARMA} models with an application to
  electricity spot prices.
\newblock \emph{Stat. Model.}, 11\penalty0 (5):\penalty0 447--470, 2011.

\bibitem[Giraitis and Robinson(2001)]{Giraitis:Robinson:2001}
L.~Giraitis and P.~M. Robinson.
\newblock Whittle estimation of {ARCH} models.
\newblock \emph{Econometric Theory}, 17\penalty0 (3):\penalty0 608--631, 2001.

\bibitem[Hannan(1973)]{Hannan73}
E.~J. Hannan.
\newblock {The asymptotic theory of linear time-series models}.
\newblock \emph{J. Appl. Probab.}, 10\penalty0 (1):\penalty0 130--145, 1973.

\bibitem[Hu and Long(2007)]{Hu:Long:2007}
Y.~Hu and H.~Long.
\newblock Parameter estimation for {O}rnstein-{U}hlenbeck processes driven by
  {$\alpha$}-stable {L}\'evy motions.
\newblock \emph{Communications on Stochastic Analysis}, 1:\penalty0 175--192,
  2007.

\bibitem[Hu and Long(2009)]{Hu:Long:2009}
Y.~Hu and H.~Long.
\newblock Least squares estimator for {O}rnstein-{U}hlenbeck processes driven
  by $\alpha$-stable motions.
\newblock \emph{Stochastic Process. Appl.}, 119:\penalty0 2465--2480, 2009.

\bibitem[Kalman(1960)]{kalman1960}
R.~E. Kalman.
\newblock A new approach to linear filtering and prediction problems.
\newblock \emph{J. Basic Eng.}, \penalty0 (82):\penalty0 35--45, 1960.

\bibitem[Kokoszka and Taqqu(1996)]{Kokoszka:Taqqu:1996}
P.~S. Kokoszka and M.~S. Taqqu.
\newblock Parameter estimation for infinite variance fractional {ARIMA}.
\newblock \emph{Ann. Statist.}, 24\penalty0 (5):\penalty0 1880--1913, 1996.

\bibitem[Ljungdahl and Podolskij(2020)]{Ljungdahl:Podolskij}
M.~Ljungdahl and M.~Podolskij.
\newblock Multi-dimensional parameter estimation of heavy-tailed moving
  averages.
\newblock \emph{arxiv.org/abs/2007.15301}, 2020.

\bibitem[Mikosch and Straumann(2002)]{Mikosch:Straumann:2002}
T.~Mikosch and D.~Straumann.
\newblock Whittle estimation in a heavy-tailed {$\rm GARCH(1,1)$} model.
\newblock \emph{Stochastic Process. Appl.}, 100:\penalty0 187--222, 2002.

\bibitem[Mikosch et~al.(1995)Mikosch, Gadrich, Kl\"{u}ppelberg, and
  Adler]{Mikosch:Gardrich:Klueppelberg:1995}
T.~Mikosch, T.~Gadrich, C.~Kl\"{u}ppelberg, and R.~J. Adler.
\newblock Parameter estimation for {ARMA} models with infinite variance
  innovations.
\newblock \emph{Ann. Statist.}, 23\penalty0 (1):\penalty0 305--326, 1995.

\bibitem[Samorodnitsky and Taqqu(1994)]{Samrodonitsky:Taqqu:1994}
G.~Samorodnitsky and M.~S. Taqqu.
\newblock \emph{Stable non-{G}aussian random processes}.
\newblock Chapman \& Hall, New York, 1994.

\bibitem[Sato(1999)]{Sato}
K.-I. Sato.
\newblock \emph{{L{\'e}vy processes and infinitely divisible distributions}}.
\newblock Cambridge University Press, Cambridge, 1999.

\bibitem[Schlemm and Stelzer(2012)]{QMLE}
E.~Schlemm and R.~Stelzer.
\newblock {Quasi maximum likelihood estimation for strongly mixing state space
  models and multivariate L{\'e}vy-driven CARMA processes}.
\newblock \emph{Electron. J. Stat.}, 6:\penalty0 2185--2234, 2012.

\bibitem[Whittle(1953)]{whittle1953estimation}
P.~Whittle.
\newblock {Estimation and information in stationary time series}.
\newblock \emph{Ark. Mat.}, 2\penalty0 (5):\penalty0 423--434, 1953.

\end{thebibliography}
\bibliographystyle{abbrvnat}
\end{document}